\documentclass[a4papers, 11pt]{amsart}
\usepackage{amscd, amssymb, mathrsfs, xcolor, bbm, csquotes, marginnote, comment, rotating}

\usepackage[normalem]{ulem}

\usepackage[centering,text={15.5cm,22cm}]{geometry}
\usepackage{graphicx,color}
\usepackage[all]{xy}
\usepackage{marvosym}
\usepackage{stmaryrd}
\usepackage{srcltx}
\usepackage{hyperref}
\usepackage{upgreek} 

\definecolor{shadecolor}{rgb}{1,0.9,0.7}

\newtheorem{theorem}{Theorem}[section]
\newtheorem{lemma}[theorem]{Lemma}
\newtheorem{lemma-definition}[theorem]{Lemma-Definition}
\newtheorem{proposition}[theorem]{Proposition}
\newtheorem{corollary}[theorem]{Corollary}

\theoremstyle{definition}

\newtheorem{definition}[theorem]{Definition}

\theoremstyle{remark}
\newtheorem{remark}[theorem]{Remark}

\numberwithin{equation}{section}
\numberwithin{figure}{section}

\newcommand{\LL} {\mathbb{L}}
\newcommand{\NN} {\mathbb{N}}

\newcommand{\QQ} {\mathbb{Q}}
\newcommand{\RR} {\mathbb{R}}
\newcommand{\TT} {\mathbb{T}}

\newcommand{\ZZ} {\mathbb{Z}}

\newcommand {\shC}  {\mathcal{C}}

\newcommand {\shE}  {\mathcal{E}}
\newcommand {\shF}  {\mathcal{F}}

\newcommand {\shH}  {\mathcal{H}}
\newcommand {\shHom} {\mathcal{H}\!\text{\textit{om}}}

\newcommand {\shL}  {\mathcal{L}}
\newcommand {\shM}  {\mathcal{M}}

\newcommand {\shN}  {\mathcal{N}}
\newcommand {\shO}  {\mathcal{O}}

\newcommand {\shQ}  {\mathcal{Q}}
\newcommand {\shR}  {\mathcal{R}}

\newcommand {\shT}  {\mathcal{T}}

  \newcommand {\fC}  {\mathfrak{C}}

\newcommand {\foM}  {\mathfrak{M}}  \newcommand {\fM}  {\mathfrak{M}}

\newcommand {\fom}  {\mathfrak{m}}

\newcommand {\Aut}  {\operatorname{Aut}}

\newcommand {\eps}  {\varepsilon}

\newcommand {\ev}  {\operatorname{ev}}

\newcommand {\Ext}  {\operatorname{Ext}}

\newcommand {\gp}  {{\operatorname{gp}}}

\newcommand {\Hom}  {\operatorname{Hom}}

\newcommand {\hra} {\hookrightarrow}

\newcommand {\id}  {\operatorname{id}}
\newcommand {\im}  {\operatorname{im}}

\newcommand {\kk} {\Bbbk}
\newcommand {\bk} {{\mathbf{k}}}

\newcommand {\ubk} {{\underline{\mathbf{k}}}}

\newcommand {\lcm}  {\operatorname{lcm}}

\newcommand {\lra}  {\longrightarrow}

\renewcommand{\O}  {\mathcal{O}}

\newcommand {\op}  {\operatorname}

\newcommand {\pr}  {\operatorname{pr}}

\newcommand {\ra}  {\to}

\newcommand {\sat}  {{\operatorname{sat}}}

\newcommand {\Spec} {\operatorname{Spec}}

\newcommand {\sra} {\twoheadrightarrow}

\newcommand {\Trop}  {\operatorname{Trop}}

\def\mydate{\ifcase\month \or January\or February\or March\or
April\or May\or June\or July\or August\or September\or October\or 
November\or December\fi \space\number\day,\space\number\year}

\newcommand{\uC}{\underline{\shC}}
\newcommand{\uD}{\underline{D}}

\newcommand{\uX}{\underline{X}}

\newcommand{\uS}{\underline{S}}

\newcommand{\uT}{\underline{T}}

\newcommand{\oM}{\overline{\mathcal{M}}}

\newcommand{\ka }{{\alpha}}

\newcommand{\cF}{\mathcal{F}}
\newcommand{\cL}{\mathcal{L}}
\newcommand{\cM}{\mathcal{M}}

\newcommand{\cO}{\mathcal{O}}
\newcommand{\cT}{\mathcal{T}}

\newcommand{\Spl}{\cL og_{\bk} ^{\mathrm{spl}}}

\newcommand{\PY}{\cL og _{\mathscr{M} _{g, n}}\times _{\cL og _{\ubk} } \Spl}

\newcommand{\oF}{\overline{\shF}}
\newcommand{\oQ}{\overline{\shQ}}
\newcommand{\fB}{\mathfrak{B}}
\newcommand{\spl}{\mathrm{spl}}

\newcommand{\ot}{\otimes}

\newcommand{\one}{\mathbbm{1}}

\newcommand{\tGamma}{{\tilde{\Gamma}}}
\newcommand{\mff}{\mathfrak{f}}
\newcommand{\mfs}{\mathfrak{s}}
\newcommand{\bGamma}{\bar{\Gamma}}
\newcommand{\fMG}{\fM _{\Gamma}}
\newcommand{\tOmega}{\tilde{\Omega}}

\newcommand{\cA}{\mathcal{A}}

\newcommand{\cC}{\mathcal{C}}
\newcommand{\cE}{\mathcal{E}}
\newcommand{\Kom}{\mathrm{Kom}}
\newcommand{\bfk}{\mathbf{k}} 

\newcommand{\ushC}{\underline{\cC}}

\begin{document}


\title
[The degeneration formula for stable log maps]
{The degeneration formula for stable log maps}
\author{Bumsig Kim, Hyenho Lho, Helge Ruddat}

\begin{abstract} 
We give a direct proof for the degeneration formula of Gromov-Witten invariants including its cycle version
for degenerations with smooth singular locus in the setting of stable log maps of Abramovich-Chen,
Chen, Gross-Siebert.
\end{abstract}

\address{\tiny School of Mathematics, Korea Institute for Advanced Study,
85 Hoegiro Dongdaemun-gu, Seoul 02455, Republic of Korea}
\email{bumsig@kias.re.kr}
\address{\tiny Department of Mathematics, CNU, 99 Daehak-ro, Yuseong-gu, Daejeon, 34134, Republic of Korea}
\email{hyenho@cnu.ac.kr}
\address{\tiny JGU Mainz,
Institut f\"ur Mathematik,
55128 Mainz, 
Germany}
\thanks{B.K. was supported by KIAS individual grant MG016404. H.R. was supported by DFG Emmy-Noether grant RU 1629/4-1. H.L. was supported by the grant ERC-2012-AdG-320368-MCSK}
\email{ruddat@uni-mainz.de}

\date{\today}

\maketitle
\setcounter{tocdepth}{1}
\tableofcontents

\section*{Introduction}
Gromov-Witten invariants are constant in smooth families and more generally in log smooth families if one considers \emph{logarithmic Gromov-Witten invariants} instead \cite{AC,Ch,GS},\cite[Thm. A.3]{MR}.
A one-parameter normal crossing degeneration, also known as \emph{semi-stable degeneration}, is such a log smooth family. 
We here consider the case where the central fibre $X$ consists of only two smooth irreducible components $X_1,X_2$ that meet in a smooth divisor $D$.
In this case, the log Gromov-Witten invariants of $X$ decompose into log Gromov-Witten invariants on the components with log structure given by the divisor $D$ respectively.
This result is the so-called \emph{degeneration formula} that was discovered and proven in pioneering works with the framework of \emph{expanded relative stable maps}: in the symplectic geometry setup by A.-M. Li and Y. Ruan \cite{LR}, 
by E. Ionel and T. Parker \cite{IP};
in algebraic geometry by J. Li \cite{Li} and D. Abramovich and B. Fantechi \cite{AF}. Q. Chen \cite{ChE} proved a hybrid version using stable log maps in the sense of \cite{ChE, Ki}. 
All of these results use target expansions. 
We give a proof in Theorem~\ref{mainthm-cycle} and Theorem~\ref{mainthm-numeric} below that goes without expansions.
The result itself is not new as it follows via comparison theorems \cite{AMW} from the prior works, but we decided to compose a direct proof in order to facilitate the arguments in \cite{GGR,GRZ}. The splitting stack in \S\ref{sec-splitting-stack} is novel.
Our gluing result of \S\ref{gluing} has been used in \cite{Bou,Bou2,Gr}. 
We give detailed arguments for the comparison results of virtual classes
by proving the commutativity of the relevant maps of triangles, see \eqref{comp_dist}, \eqref{two_relatives}.
Novel is the elaboration of the tropical point of view for the degeneration formula inspired by \cite{NS,MR,MR2}.
The tropical point of view in log Gromov-Witten theory was first established in \cite{GS}.

A decomposition formula for general log smooth fibres has been given in \cite{ACGS}. A symplectic geometry approach has been followed in \cite{Pa,Te} with a more general degeneration formula in \cite{Pa2}. More general gluing formulae in log geometry has been obtained in \cite{ACGS2,Wu} and a degeneration formula in \cite{Ran}.

\subsection{Acknowledgments}
We thank Michel van Garrel and Tim Gräfnitz for advice on the presentation. Mark Gross pointed out an example that shows that the seemingly mild generalization of the degeneration formula to degenerations of relative invariants is just as involved as a generalization to general normal crossing degenerations. We also thank Yuki Hirano and Tom Graber for helpful technical advice.

\subsection{Conventions} 
We refer to \cite{K.Kato} for the basics of log geometry. All log schemes will be fine and saturated and we denote them by undecorated letters like $S$. 
We refer to the underlying scheme by $\uS$ and occasionally, by abuse of notation, we also refer to $\uS$ as the scheme with trivial log structure.
For $\uD\subset \uX$ a subvariety we denote the pullback of the log structure from $X$ to $D$ by $\shM_X|_D$.
We use $\shM$ to refer to monoid sheaves and $\mathscr{M}$ to refer to moduli stacks, e.g.,~$\mathscr{M}_{g,n}(X/B,\beta)$ 
denotes the moduli stack of $n$-marked basic stable log maps of genus $g$ and class $\beta$ to a target log space $X$ that is log smooth over $B$. 
We will sometimes use the notation $\mathscr{M}_{g,S}(X/B,\beta)$ for some finite set $S$ that is used to label the markings of the stable maps. 
With few exceptions clear from the context, curves for us will be connected.
Out of the $n$ markings, some may have prescribed contact orders to strata in $X$ and this is a part of the data of $\beta$.
For a monoid $M$, we denote its Grothendieck group by $M^\gp$, similarly for sheaves of monoids. 
We set $M^\vee:=\Hom(M,\NN)$, denote by $M[1]$ the set of generators of dimension one faces of $M$ and for $m\in M$, we set $m^\perp=\{n\in M^\vee| n(m)=0\}$. 
For a graph $\Gamma$, we let $E(\Gamma)$ denote the set of its edges. 
We work over a fixed field $\kk$ of characteristic zero. When we refer to \emph{a point}, it will be implicit that we mean a geometric point.

\section{Geometric setup and the main result}  \label{sec-setup-results}
\subsection{Semi-stable degenerations}
\label{sec-semistable}
Consider a semi-stable degeneration $\pi:\frak{X}\ra B$, i.e., a projective surjection from a smooth scheme $\frak{X}$ to a smooth one-dimensional scheme $B$ with smooth fibres away from a point $b$ and
$X=\pi^{-1}(b)$ is simple normal crossings. We assume $X$ consists of two smooth components $X_1,X_2$ that meet in a smooth subvariety $D$ that is a divisor in each of $X_1,X_2$. 

The divisor $X\subset\frak{X}$ defines a divisorial log structure on $\frak{X}$, concretely
it is given by the monoid sheaf $\shM_{\frak{X}}:=\shO^\times_{\frak{X}\setminus X}\cap \shO_{\frak{X}}$ with its inclusion in $\shO_{\frak{X}}$.
We analogously obtain a divisorial log structure on $B$ by $\shM_B:=\O^\times_{B\setminus \{b\}}\cap \shO_B$ that maps into $\shM_\frak{X}$ under $\pi^*$, so we turned $\pi$ into a log map which is in fact log smooth, even over $b$. By \cite{BF} and \cite[Theorem A.3]{MR} the log Gromov-Witten invariants of all fibres of $\pi$ agree. 
The main purpose of the degeneration formula is to compute these invariants on the special fibre $X$. 
Henceforth, we will therefore forget about $\pi$ and only consider a log smooth $X\ra b$ that is decomposed as in this degeneration.

\subsection{Log smooth target $X$}
\label{subsec-setup}
We let $\bk:=\Spec (\NN\stackrel{\tiny 1\mapsto 0}{\lra} \kk)$ denote the standard log point. 
(This can be thought of as $b$ above and it now comes with a distinguished chart.) 
We denote its underlying point scheme by $\ubk=\Spec \kk$.

Throughout, we fix a log smooth morphism $X\ra \bk$ where the underlying scheme decomposes as $\uX = \uX_1\sqcup _{\underline D} \uX_2$ in smooth irreducible components $\uX_i$ and $\underline D$ is the smooth connected singular locus of $\uX$. We assume that the log structure is of \emph{semi-stable type} which means that 
$X\ra \bk$ is strict away from $D$ and the stalks of the \emph{characteristic} $\overline\shM_{{X}}:=\shM_{{X}}/\shO_{{X}}$ are isomorphic to $\NN^2$ at points in $D$.
Unwinding the definitions we obtain the following standard fact.
\begin{lemma} 
\label{DF-log} 
Let $\uX = \uX_1\sqcup _{\underline D} \uX_2$ be a scheme over $\kk$. 
Giving a log smooth morphism $X\ra \bk$ of semi-stable type with underlying variety $\uX$ is equivalent to giving two line bundles $\shL_1,\shL_2$ on $\uX$ together with maps $s_i:\shL_i\ra\shO_X$ and a global section $\pi\in\Gamma(\uX,\shL_1\otimes \shL_2)$ such that 
\begin{enumerate}
\item $\shL_1|_{X_2}\stackrel{s_1|_{X_2}}{\lra}\shO_{X_2}$ is injective and identifies $\shL_1|_{X_2} = \shO_{X_2}(-D)$, and similarly with indices 1,2 interchanged,
\item $\pi$ trivializes $\shL_1\otimes \shL_2\cong\shO_X$ and
\item $(s_1\otimes s_2)(\pi)=0$.
\end{enumerate}
\end{lemma}


\begin{remark}
If $X$ is the central fibre of a family $\pi:\frak{X}\ra B$ as before, 
then we find $\shL_i=\shO_{\frak{X}}(-X_i)|_X$ with $s_i$ the restriction of the inclusion $\shO_{\frak{X}}(-X_i)\hra \shO_{\frak{X}}$ to $X$ and 
$\pi$ defines a section of $\shO_{\frak{X}}(-X_1-X_2)$ over an \'etale neighbourhood of $b$, so indeed $\pi\in\Gamma(\uX,\shL_1\otimes \shL_2)$.
\end{remark}

\begin{remark} 
A scheme of the form $\uX = \uX_1\sqcup _{\underline D} \uX_2$ permits a lift to a log smooth morphism $X\ra \bk$ if and only if $\shT^1:=\Ext(\Omega_{\uX},\shO_{\uX})$ has a nowhere vanishing section. More generally, if $\shT^1$ has a section with smooth zero locus then $\uX$ can be upgraded to a log toroidal morphism $X\ra \bk$, see \cite{FFR}.
\end{remark}

We denote by $\iota_i: \uX_i\ra \uX$ the natural inclusions.
Using the natural surjection
\begin{equation}
\label{global-generation}
\NN^2\ra\oM_X:=\shM_X/\shO^\times_X,\qquad e_i\mapsto (\hbox{local equation for }X_i\hbox{ in }\frak{X})
\end{equation}
we will later make use of the identification 
\begin{equation}
\label{decomposeMbar}
\oM_X=\iota_{1,*}\NN\oplus\iota_{2,*}\NN.
\end{equation}

There is a natural surjection onto $\shM_X$ from the following sheaf of monoids
$$\{(n_1,n_2,f)|n_1,n_2\ge 0,f\hbox{ is a local generator for }\shL_1^{\otimes n_1}\otimes\shL_2^{\otimes n_2}\},$$
see for instance Complement 1 in \cite{K.Kato}.

\begin{remark} \label{rem-log-Xi}
Note that $\uX_1$ has two different natural log structures namely $\shM_X|_{X_1}$, the restriction from $X$, and the divisorial log structure from $D$, i.e., $\shM_{(X_1,D)}:=\shO^\times_{X_1\setminus D}\cap\shO_{X_1}$, similarly for $\uX_2$.
For the remainder of the paper, we use $X_i$ to refer to the latter one, i.e., $X_i=(\uX_i,\shM_{(X_1,D)})$.
There is a natural inclusion $\shM_{(X_1,D)}\subset \shM_X|_{X_1}$ compatible with the maps to $\shO_{X_1}$ because 
$\shM_{(X_1,D)}$ is the log structure associated to the submonoid sheaf
$$\{(0,n_2,f)|n_2\ge 0,f\hbox{ is a local generator for } \shL_1^{ \otimes 0}|_{X_1}\otimes \shL_2^{\otimes n_1}|_{X_1}\}$$
since by Lemma~\ref{DF-log}-(1) we have $\shL_2|_{X_1}=\shO_{X_1}(-D)$. Hence, we have a map of log schemes
$$(\uX_1,\shM_X|_{X_1})\ra (\uX_1,\shM_{(X_1,D)})$$
and similarly for $\uX_2$ and this difference is what causes most of the work in later chapters. The induced inclusion $\oM_{(X_1,D)}\subset \oM_X|_{X_1}$ is $\{0\}\oplus\NN_D\subset \NN_{X_1}\oplus\NN_D$ given by the exponents $n_1,n_2$.
\end{remark}

\subsection{Cycle version of the degeneration formula} \label{subsec-cycle-degen-formula}
We fix an effective curve class $\beta\in H_2(\uX)$.
We consider in \S\ref{sec-graphs} certain decorated bipartite graphs $\Gamma$. 
Bipartite means that there is a given map $r:\{\hbox{vertices of }\Gamma\}\ra\{1,2\}$ and the vertices of each edge have different values under $r$.
To each vertex $V$ of $\Gamma$ we associate a moduli stack $\mathscr{M}_V$ that classifies stable log maps to $X_{r(V)}$ governed by data from $\Gamma$ (see Theorem~\ref{mainthm-numeric} and \S\ref{sec-graphs}, \S\ref{section-ob-theories} for more details).
Here, $X_1$ carries the divisorial log structure via the divisor $\uD$, similarly for $X_2$. 
The adjacent edges at $V$ index marked points that map to $\uD$, so there is an evaluation map $\mathscr{M}_V\ra \prod_{e\ni V}\uD$ where the product is over the edges of $\Gamma$ that contain $V$.
Since, by usual conventions, markings ought to be ordered, we also need to keep track of an ordering of the edges of $\Gamma$ and we denote this edge-ordered graph $\tGamma$.
We define $\bigodot _V \mathscr{M}_V $ by the Cartesian square
\begin{equation} \label{gluediag}
\vcenter{ \xymatrix{ 
\bigodot _V \mathscr{M}_V \ar^u[r]\ar[d]& \prod_V \mathscr{M}_V\ar[d]\\
\prod_{e}\uD \ar^(.4)\Delta[r]& \prod_V\prod_{e\ni V} \uD\\
} }
\end{equation}
where the bottom left is the product over all edges of $\tGamma$ and the map $\Delta$ 
is $(d_e)_e\mapsto (d_e)_{V\in e}$, that is
on the factor $\uD$ indexed by an edge $e$ it is the diagonal into the two components indexed by the vertices of $e$ that appear in the bottom right.
This diagram has the effect that the stable maps in the $\mathscr{M}_V$ for various $V$ are glued over their evaluations in $\uD$ as prescribed by $\tGamma$ to form a stable map to $\uX$ that is then an object in $\bigodot _V \mathscr{M}_V $. 
To further garnish this stable map with a compatible log structure to get a stable log map to $X$, a finite choice is to be made. 
In fact, there is an \'etale map ${\phi_\tGamma}:\mathscr{M}_\tGamma\ra \bigodot _V \mathscr{M}_V $ where objects in $\mathscr{M}_\tGamma$ are stable log maps to $X$ whose dual intersection graph collapses to $\tGamma$.
We will show that
\begin{equation} \label{degPhi}
\deg({\phi_\tGamma} )=\frac{\prod_e w_e}{l_\Gamma}
\end{equation}
for the degree of this map (see Lemma~\ref{etale},\eqref{phi} or \eqref{gluing-degree})
where $w_e$ is the contact order to $\uD$ at the relative marking corresponding to $e$ (and this is necessarily the same for $X_1$ and $X_2$) and $l_\Gamma=\lcm(\{w_e\})$. The contact order is defined to be the weight of $e$, see \eqref{node-relation} and the sentence thereafter.
We also have a natural map $F:\mathscr{M}_\tGamma\ra \mathscr{M}$ to the moduli stack $\mathscr{M}:=\mathscr{M}_{g, n} (X/\bk, \beta )$ of stable log maps to $X$ and we show in Lemma~\ref{SplFund2} that the virtual degree of $F$ is $\frac{|E(\Gamma)|!}
{l_\Gamma}$ where $E(\Gamma)$ is the set of edges of $\Gamma$. For every $\tGamma$, we have a commutative diagram
\begin{equation}\label{diag-moduli-spaces}
\vcenter{\xymatrix{ 
\mathscr{M}_\tGamma\ar^F[r]\ar_{\phi_\tGamma}[d]& \mathscr{M} \ar^{\op{ev}}[dr]\\
\bigodot _V \mathscr{M}_V \ar^u[r]& \prod_V \mathscr{M}_V \ar^{\op{ev}}[r] &\uX^n\\
} }
\end{equation}
where $\op{ev}$ denotes respectively the evaluation map for the $n$ marked points.
The following is the main result and will be proved at the end of \S\ref{mainproof}.
\begin{theorem}[Cycle version of the degeneration formula] 
\label{mainthm-cycle} 
We have
$$\llbracket\mathscr{M}\rrbracket = \sum_{\tGamma} \frac{l_\Gamma}{|E(\Gamma)|!} F_*\phi^*\Delta^! \prod_V\llbracket\mathscr{M}_V\rrbracket$$
where $\phi=\phi_\tGamma$ and $\llbracket\mathscr{M}\rrbracket$ is the natural virtual fundamental class for $\mathscr{M}$ and similarly $\prod_V\llbracket\mathscr{M}_V\rrbracket$ is the outer product of the natural virtual fundamental classes for $\mathscr{M}_V$.
\end{theorem}

\subsection{Numerical degeneration formula}
Let us deduce from Theorem~\ref{mainthm-cycle} the numerical version of the degeneration formula. 
Assume we are given an operational Chow class $\gamma\in A^{*}(X^n)$ and an operational Chow class $\psi\in A^{*}(\prod_V \mathscr{M}_V)$ 
whose pullback to $\mathscr{M}_\tGamma$ comes from an operational Chow class $\psi'\in A^{*}(\mathscr{M})$. 
Recall that taking degree is proper push-forward to a point and thus compatible with finite maps. Inserting $\gamma$ and $\psi'$ into
Theorem~\ref{mainthm-cycle} gives
\begin{align}
\label{deg-degen-form}
\begin{split}
\deg\left(\psi'\cap (\gamma\cap\llbracket\mathscr{M}\rrbracket) \right)&=\sum_\tGamma 
\frac{l_\Gamma}{|E(\Gamma)|!} \deg\left(\psi'\cap \big(\gamma\cap F_*\phi^*\Delta^! \prod_V\llbracket\mathscr{M}_V\rrbracket\big) \right)\\
&=\sum_\tGamma 
\frac{l_\Gamma}{|E(\Gamma)|!} \deg\left(\psi\cap \big(\gamma\cap \phi^*\Delta^! \prod_V\llbracket\mathscr{M}_V\rrbracket\big) \right)\\
&\overset{\eqref{degPhi}}{=}\sum_\tGamma  \frac{\prod_e w_e}{|E(\Gamma)|!}
\deg\left(\psi\cap \big(\gamma\cap \Delta^!\prod_V\llbracket\mathscr{M}_V\rrbracket \big)\right).
\end{split}
\end{align}
Here the last equality uses that $\phi_*\phi^*$ is multiplication by $\deg({\phi} )$.

 The expressions in \eqref{deg-degen-form} may be reinterpreted in Borel--Moore homology instead. In this case, read $\ev^*(\gamma)\cap$ for each occurrence of $\gamma \cap$ above, $\gamma,\psi$ are cohomology classes now and we apply the cycle map $A_*\ra H_{2*}$ to all occurrences of $\llbracket\mathscr{M}_V\rrbracket$ above. Then \eqref{deg-degen-form} holds with these reinterpretations.
The advantage of the latter interpretation is that we can impose incidence at an arbitrary cocycle $\gamma\in H^*(X^n)$ at the cost of signs in the following.

Let $\{ \delta _j^{1}\}_{j}$ be a homogeneous basis of $H^*(\uD, \QQ)$ 
and let $\{ \delta _j ^{2}\}_{j}$ be the dual basis 
in the sense that 
\[ \int _{\uD} \delta _i ^{2} \delta _j ^{1}  =  \left\{ \begin{array}{rl}  0 & \text{ if } i\ne j \\
             1 & \text{ if } i=j , \end{array} \right.  \]
where 2 is purposefully before 1 to have no signs in the representation of the diagonal. 
Define the sign $(-1)^{\eps}$ by the equality
$$\prod^n_{i=1} \gamma_i \prod_{e} \delta^1_{e,j_e}\otimes\delta^2_{e,j_e} = (-1)^{\eps} \prod_V\left(\prod_{i\in n_V}\gamma_i\right)\left(\prod_{e\ni V} \delta^{r(V)}_{e,j_e}\right). $$
Then, we conclude from Theorem~\ref{mainthm-cycle} the following result.
\begin{theorem}[Numerical version of the degeneration formula] \label{mainthm-numeric} 
For $\gamma _i \in H^*(\uX, \QQ)$ and non-negative integers $m_i$, in Witten's correlator-notation  where $\tau_m(\gamma)$ means $\psi^m\ev^*(\gamma)$, we have
\begin{align*}       
\left\langle \prod_{i=1}^n\tau_{m_i}(\gamma_i) \right\rangle_{g,\beta}^X 
     &= \sum _{\tGamma}  \sum _{(j_e)_e } 
      \frac{\Pi _{e}  w_e }{|E(\tGamma )|!}(-1)^\eps      \prod _{V\in V(\Gamma)}  
    \left\langle \left.\prod_{i\in n_V}\tau_{m_i}(\iota^*_{r(V)}\gamma_i)\right|
\prod_{e\in V} \tau_0(\delta^{r(V)}_{j_e})   \right\rangle_{g_V,\beta_V}^{X_{r(V)}} 
\end{align*}
where $\eps$ is determined as before, the first sum runs over all $\tGamma\in\tOmega (g, n, \beta)$ as introduced in \S\ref{sec-graphs} (see also for $n_V,\beta_V$) and
the second sum runs over all tuples in
$\{1, ..., \mathrm{rk}H^*(\uD) \}^{E(\tGamma )}$.
The moduli stack underlying the left hand side is $\mathscr{M}_{g, n} (X/\bk, \beta )$ and that for the right hand side is 
$\mathscr{M}_V := \mathscr{M}_{g_V, n_V\cup E_V} (X_{r(V)}/\ubk, \beta_V )$
where $E_V$ refers to the ordered set of edges in $\tGamma$ adjacent to $V$. The positive contact orders $w_e$ to $D$ for $e\in E_V$ are 
part of the data $\beta_V$.
If $\Gamma$ has only one vertex, then we set $\Pi _e w_e / | E(\tGamma ) |! = 1$. The sum is finite (see \S\ref{sec-graphs}).
\end{theorem}

The formula is a straightforward version of the degeneration formula of \cite{Li, AF, ChE}.

\begin{remark} 
If $X=\pi^{-1}(b)$ is the central fibre of a semi-stable degeneration $\frak{X}\ra B$ as in \S \ref{sec-setup-results},  
we fix a $\hat\beta$ in $H_2(\frak{X})$ and then for $b'\in B$ and ${X_{b'}}=\pi^{-1}(b')$, we have an identity
$$\sum_\beta\left\langle \prod_{i=1}^n\tau_{m_i}(\gamma_i) \right\rangle_{g,\beta}^X=\sum_{\beta'}\left\langle \prod_{i=1}^n\tau_{m_i}(\gamma_i') \right\rangle_{g,\beta'}^{X_{b'}} $$
provided that: 
\begin{enumerate}
\item We take the sum respectively over all $\beta\in H_2({X_b})$  and $\beta'\in H_2({X_{b'}})$ which map to $\hat\beta$.
\item The classes $\gamma_i'\in H^*(X_{b'})$ and
$\gamma_i\in H^*(X_{b})$ are pullbacks from the same element in $H^*(\frak{X})$ for each $i$.
\end{enumerate}
This statement follows from \cite[Proposition 5.10]{BF} as explained in \cite[Theorem~A.3]{MR}. 
\end{remark}


\section{Graphs}
\label{sec-graphs}

Consider a bipartite graph $\Gamma$, i.e.,~we have a map $r:\{\hbox{vertices of }\Gamma\}\ra\{1,2\}$ and the vertices of each edge have different values under $r$.
Each vertex $V$ is decorated with a tuple $(g_V,\beta_V,n_V)$ with $g_V\ge 0$ called the \emph{genus}, $n_V\subset\{1,...,n\}$ and $\beta_V$ an effective curve class in $X_{r(V)}$.
Each edge $e$ is decorated with a positive integer $w_e$, called \emph{the weight}.
 Furthermore, we require $\Gamma$ to satisfy the following properties.

\begin{align}\label{class_sum} 
\hphantom{(curve class)} \qquad\qquad  \sum _{V:r(V)=1} \iota _{{1}, *} \beta _V  + \sum _{V:r(V)=2} \iota _{{2}, *} \beta _V   = \beta \qquad\qquad \hspace{-.1cm} \emph{(curve class)} 
\end{align} 
\begin{align}\label{rel_class_sum} 
 \hphantom{(contact order)}\qquad\qquad\qquad\qquad \beta _V \cdot D = \sum _{e\ni V} w_e, \qquad\qquad\qquad\qquad \hspace{-.1cm} \emph{(contact order)} \end{align}
\begin{align}\label{stability} 
\hphantom{stability}\qquad\qquad \beta _V \ne 0  \ \ \text{  if  } \ \ 2g_V + |n_V| + \mathrm{val}(V) < 3, \qquad\qquad \hspace{.3cm} \emph{(stability)} \end{align}
 \begin{align}\label{genus}    
 \hphantom{(genus)}\qquad\qquad\qquad\qquad  1- \chi_{\mathrm{top}} (\Gamma) +  \sum _{V} g_V   = g, \qquad\qquad\qquad\qquad  \hspace{-.68cm}\emph{(genus)} \end{align}
 \begin{align}\label{partition_n}  
 \hphantom{(markings)}\qquad\qquad\qquad\qquad\coprod _V n_V = \{ 1, ..., n\}  \qquad\qquad\qquad\qquad \hspace{-.2cm} \emph{(markings)} \end{align}
We call $\Gamma$ of type $(g, n, \beta)$ if it satisfies these conditions and denote the set of all such $\Gamma$ up to isomorphism by $\Omega(g, n, \beta)$. 
The set $\Omega (g, n, \beta)$ is a finite set.
Indeed, $\eqref{stability}$, $\eqref{genus}$ and $\eqref{partition_n}$ imply that the set of marking and genus decorated graphs is finite and then since the trivial curve class is indecomposable in the cone of effective curve classes, the finiteness of $\Omega(g, n, \beta)$ follows.

We denote by $\tGamma$ a decorated graph $\Gamma$ as above that is additionally equipped with edge markings, i.e., with a bijection $E(\Gamma)\cong\{ e_1, ..., e_{|E(\Gamma )|} \}$ and here the $e_i$ are formal symbols. 
Let $\tilde\Omega(g, n, \beta)$ denote the set of all such $\tilde\Gamma$ up to isomorphism. 
Let $\Aut(\Gamma)$ denote the (finite) group of automorphisms of $\Gamma$ that are compatible with the decorations.
Note that
\begin{equation} \label{Gamma-tildeGamma}
\sum_{\Gamma\in \Omega(g, n, \beta) }\frac1{|\Aut(\Gamma)|} = \frac{|\tilde \Omega(g, n, \beta)|}{|E(\Gamma )|!}.
\end{equation}
Given $\tGamma$ as above, we denote by $\Gamma_i$ the subgraph with the vertex set $\{V: r(V)=i\}$
and we keep the adjacent edges as half-edges, 
Each adjacent edge is considered to have only one vertex, topologically $[0,\infty)$.
We carry over the decorations to the vertices and half-edges: $\beta_V,g_V,n_V,w_e$ and the ordering of the half-edges.
We then denote by $\Gamma_V$ the connected component of $\Gamma_1$ or $\Gamma_2$ containing the vertex $V$.


\section{Stable log maps}
We refer to \cite{K.Kato} for the basics on log geometry and to \cite{GS,AC,Ch} for the basics of stable log maps that we recall here now. Note that \emph{smooth} means \emph{log smooth} in the context of log schemes. 
Let $Y,W$ be log schemes with log structures coming from the Zariski site and let $Y\ra W$ be a smooth and projective morphism. 
We are going to apply this to $X\ra\bk$ and $X_i\ra\ubk$ later on; see the beginning of \S\ref{subsec-setup} for the notations $\bk$, $\ubk$.
We recall Definitions\,1.3 and 1.6 from \cite{GS}.
\begin{definition} \label{def-stablemap}
A \emph{prestable log map} is a commutative diagram of log morphisms
\begin{equation}\label{square} \begin{CD} C @>f>> Y \\
        @V{\pi}VV @VVV \\
        S @>{f_S}>> W \end{CD}
\end{equation}
such that $\pi$ is smooth and integral and the fibres of $\underline{\pi}:\underline C\ra\uS$ are reduced and connected curves.
There are sections $x_1,...,x_n:\uS\ra \underline C$ for the marked points with mutually disjoint images and these images are precisely the locus in the complement of nodes of fibres where  $\pi$ is not strict. By Theorem 1.3 in \cite{F.Kato}, away from the nodes, $\oM_C=\pi^*\oM_S\oplus\bigoplus_i x_{i,*}\NN$.
A prestable log map is \emph{stable} if the diagram of underlying schemes constitutes a stable map.
\end{definition}
Consider a stable log map with $S$ a point, $Q:=\oM_S$ and $e\in C$ a node, then $\oM_{C,e}=Q\oplus_\NN\NN^2$ for some $\NN\ra Q,1\mapsto q_e\neq 0$.
Let $\eta_1,\eta_2$ be the generic points of the components adjacent to $e$ in $C$, the map $f$ together with the generizations $e\to \overline{\eta_i}$ induce a commutative diagram (see Discussion 1.8, p.459 in \cite{GS})
\begin{equation} \label{butterfly}
\begin{split}
\xymatrix@C=20pt
{&P_{\eta_1}\ar^{f_{\eta_1}}[rr]&&Q\\
P_e\ar^{f_e}[rr]\ar^{\chi_1}[ru]\ar_{\chi_2}[rd]
&&Q\oplus_\NN\NN^2 \ar[ru]\ar[rd]\ar@{^{(}->}[r]
&\ar[u]_{\pr_1} \ar[d]^{\pr_2} Q\times Q\\
&P_{\eta_2}\ar_{f_{\eta_2}}[rr]&&Q
}
\end{split}
\end{equation}
where $P_e:=\oM_{Y,f(e)}$, $P_{\eta_i}:=\oM_{Y,f(\eta_i)}$ and the horizontal maps are induced  by $f$.
The diagram defines a map $u_e:P_e\ra\ZZ$ by the property 
\begin{equation}
\label{node-relation}
f_{\eta_2}\circ{\chi_2}-f_{\eta_1}\circ{\chi_1}=u_e\cdot q_e.
\end{equation}
If $u_e$ is non-zero, there is a unique primitive $\tilde u_e\in \Hom(P_e^\gp,\ZZ)$ and $w_e>0$ such that $u_e=w_e \tilde u_e$. 
We call $w_e$ \emph{the weight} of $e$. If $u_e=0$, set $w_e=0$.
For a monoid $P$, define $P^\vee=\Hom(P,\NN)$. Consider the monoid
\begin{equation}
\label{eq-basic}
Q^\vee_{\op{basic}} := \left\{ \left. ((V_\eta)_\eta,(l_e)_e)\in\bigoplus_\eta P^\vee_\eta\oplus\bigoplus_e\NN\,\right|\, V_{\eta_2}\circ\chi_2 -V_{\eta_1}\circ\chi_1=l_eu_e \hbox{ for all }e\right\}
\end{equation}
where the first sum runs over the generic points $\eta$ of irreducible components of $C$ and the second sum runs over the nodes $e$.
\begin{definition}
\label{def-basic} 
If $S$ is a point and $Q=\oM_S$ as before, for $\eta$ the generic point of a component of $C$, let $f^\vee_\eta:Q^\vee\ra P_\eta^\vee$ denote the dual of $f_\eta$.
For $e$ a node of $C$, let $q_e^\vee:Q^\vee\ra\NN$ be the evaluation of an element of $Q^\vee$ on $q_e$.
The tuple $((f^\vee_\eta)_\eta,(q_e^\vee)_e)$ gives a well-defined \emph{structure map} 
$$Q^\vee\ra Q^\vee_{\op{basic}}$$ 
because the image $\big((f^\vee_\eta(q))_\eta,(q_e^\vee(q))_e\big)=:((V_\eta)_\eta,(l_e)_e)$ of every $q\in Q^\vee$ satisfies the relation 
$V_{\eta_2}\circ\chi_2 -V_{\eta_1}\circ\chi_1=l_eu_e$ for each $e$ in the definition of $Q^\vee_{\op{basic}}$ due to \eqref{node-relation}.

We call the stable log map $f:C/S\ra Y/W$ \emph{basic} if the structure map $Q^\vee\ra Q^\vee_{\op{basic}}$ is an isomorphism. 
A stable log map with more general base $S$ is basic if its restriction to all points in $S$ is basic.
\end{definition}

If $f$ is basic and $\rho\in Q^\vee_{\op{basic}}$ an element, say $\rho=((V_\eta)_\eta,(l_e)_e)$, then Definition~\ref{def-basic} implies
$$ \rho(q_e)=l_e.$$

If $x_i:S\ra C$ is one of the sections of a stable log map $f:C/S\ra Y/W$ with $S$ a point, then we denote $P_{x_i}:=\oM_{Y,f(x_i(S))}$ and $f$ induces a map
$$u_i\colon P_{x_i}\ra \oM_{C,x_i(S)}=Q\oplus \NN \ra \NN$$
where the second map is the projection to the second summand. The map constitutes an element $u_i\in P_{x_i}^\vee$ which we call the \emph{contact order} of $f$ at $x_1$.
\begin{definition} 
\label{def-beta}
A \emph{class} $\beta$ of stable log map to $Y/W$ consists of 
\begin{itemize}
\item an element of $H_2(Y)$ that we also call $\beta$, 
\item a genus $g\ge 0$, 
\item a number of markings $n\ge 0$ and
\item for $1\le i\le n$, a strict closed embeddings $Z_i\subset Y$ and section $s_i\in\Gamma(Z_i,\shHom(\oM_{Z_i}^\gp,\ZZ))$ that does not extend to any closed subset of $Y$ that is strictly larger than $Z_i$.  
\end{itemize}
We say that a stable log map $f$ is of class $\beta$ if the underlying stable map is of genus $g$, of class $\beta$ with $n$ markings and if the contact order $u_i$ at $x_i$ agrees with $s_i$ over every point in $S$.
\end{definition}
We denote the moduli stack of basic stable log maps of class $\beta$ by $\mathscr{M}_{g,n}(Y/W,\beta)$.
This stack is the source of a forgetful functor to the another stack $\cL og_{\mathscr{M}_{g,n}}$ that we recall in \S\ref{sec-splitting-stack}.
Moreover, we have a commutative square
\[ 
\xymatrix{
\mathscr{V}\ar_{\pi}[d]\ar^f[r] & Y\ar[d]\\
\mathscr{M}_{g,n}(Y/W,\beta) \ar[r]&W
} 
\]
where the left vertical arrow denotes the universal family and the top horizontal arrow is the universal map.
Let $\shT_{Y/W}$ denote the relative tangent sheaf of the log smooth map $Y\ra W$.
Similar to the construction given in \S\ref{subsec-virt-fund-cl} below, we obtain a perfect obstruction theory
$(R\pi _* f^*\shT_{Y/W})^\vee\ra\LL_{\mathscr{M}_{g,n}(Y/W,\beta)/\cL og_{\mathscr{M}_{g,n}}},$
see also \cite{GS},\,\S5.

The reader may find the general definition of \emph{combinatorial finiteness} for a class $\beta$ in \cite{GS},\,Definition~3.3. 
This condition holds in the situations of interest to us because the set $\Omega(g, n, \beta)$ that we introduced in \S\ref{sec-graphs} is finite. 

The main result of \cite{GS,AC,Ch} is then as follows.
\begin{theorem} 
If $\beta$ is combinatorially finite then $\mathscr{M}_{g,n}(Y/W,\beta)$ is a proper Deligne-Mumford stack of finite type over $W$ with natural virtual fundamental class $\llbracket\mathscr{M}_{g,n}(Y/W,\beta)\rrbracket$.
\end{theorem}
We will consider $\mathscr{M}:=\mathscr{M}_{g,n}(X/\bk,\beta)$ as well as $\mathscr{M}_V:=\mathscr{M}_{g_V,n_V\cup E_V}(X_i/\ubk,\beta_V)$ for certain $\beta_V$ in \S\ref{section-ob-theories}.


\section{From curves to graphs and tropical curves}
\label{section-graphs-to-curves} 
Transferring notation from the previous section, we now set $Y:=X$, $W:=\bk$. 
By Definition~\ref{def-stablemap}, the characteristic of the log structure at every point $\bar s\in \mathscr{M}(X/\bk,\beta)$ is given by the dual of \eqref{eq-basic}, that is, $\oM_{\bar s}=Q_{\op{basic}}:=(Q^\vee_{\op{basic}})^\vee$. 
By the description of $\oM_X$ in \eqref{decomposeMbar}, we have $P_\eta\cong \NN^2$ if and only if $\eta$ maps to $D$ and $P_\eta\cong \NN$ otherwise. A similar statement holds for $P_e$.
By definition, $Q^\vee_{\op{basic}}$ is a saturated submonoid of $(\bigoplus_\eta P_\eta^\vee)\oplus(\bigoplus_e\NN)$ and the subgroup of invertible elements of the latter is trivial. Applying $\Hom(\cdot,\NN)$ to this inclusion, we obtain a map 
\begin{equation}
\label{basic-generation-surjection}
\left(\bigoplus_\eta P_\eta\right)\oplus\left(\bigoplus_e\NN\right) \ra Q_{\op{basic}}
\end{equation}
that is surjective up to saturation, i.e., for every $q\in Q_{\op{basic}}$ there is a $k>0$ such that $kq$ is in the image.
The generator of the $\NN$-summand for $e$ maps as $1\mapsto q_e$ and we denote the restriction of \eqref{basic-generation-surjection} to the $P_\eta$-summand by $V_\eta:P_\eta\ra Q_{\op{basic}}$. This notation is compatible with the notation in \eqref{eq-basic} because, for an element $\rho\in Q_{\op{basic}}^\vee$, composing $V_\eta$ with evaluation on $\rho$ yields the component that is called $V_\eta$ in \eqref{eq-basic}, see \eqref{tropical-sequence}.

Let $\one$ denote the generator of $\NN$ in the log structure of the standard log point $\bk$. 
The section $\one$ maps to every stalk in all the log structures of the schemes in \eqref{square} and we call them $\one$ also in these other places. Note that $\one\neq 0$ in all places by the locality of monoid maps induced from log morphisms.

For $\eta$ a generic point of a component of $C$, in light of \eqref{global-generation} and \eqref{decomposeMbar}, consider the composition
\begin{equation}
\label{tropical-sequence-0}
\NN^2=\Gamma(X,\oM_X)\sra P_\eta \stackrel{V_\eta}{\lra} Q_{\op{basic}},\qquad e_i \mapsto V_\eta(e_i).
\end{equation}
Note that $\one=(1,1)$ on the left maps to the element $\one$ on the right independent of $\eta$ because this only depends on the bottom horizontal map in \eqref{square} which on log charts is given by $\NN\ra  Q_{\op{basic}}, 1\mapsto\one$. 
Therefore, $\one=V_\eta(1,1)$ for all $\eta$ and hence by \eqref{node-relation}
\begin{equation}
\label{ue-one-zero}
u_e(\one)=0
\end{equation}
for all nodes $e$ and thus $\tilde u_e=(1,-1)$ or $\tilde u_e=(-1,1)$ whenever it is non-zero. 
Here we implicitly represent $\tilde u_e:P_e\ra \ZZ$ via the composition with $\NN^2\sra P_e$.

\begin{lemma} 
\label{lem-identity-one}
For every edge $e\in E(\Gamma_C)$ there is an labelling $\eta_1,\eta_2$ of the generic points of adjacent curve components so that we have an identity of elements in $Q_{\op{basic}}$ of the form $$ V_{\eta_1}(e_1) + w_e q_e + V_{\eta_2}(e_2) =\one. $$
\end{lemma}
\begin{proof} 
The statement follows from combining \eqref{node-relation} with the identity $V_{\eta_2}(e_1)+V_{\eta_2}(e_2)=\one$.
\end{proof}

Next, assume we are given an element $\rho=((V^\rho_\eta)_\eta,(l_e)_e)\in Q^\vee_{\op{basic}}$. Consider the composition
\begin{equation}
\label{tropical-sequence}
\NN^2=\Gamma(X,\oM_X)\sra P_\eta \stackrel{V_\eta}{\lra} Q_{\op{basic}}\stackrel{\rho}{\lra} \NN,\qquad e_i \mapsto \rho(V_\eta(e_i))=V^\rho_\eta(e_i)
\end{equation}
We set $l:=\rho(\one)=\rho(V_\eta(1,1))$ which is independent of $\eta$ by the commutativity of \eqref{square}.
Applying $\Hom(\cdot,\RR_{\ge0})$ to the sequence \eqref{tropical-sequence} yields a map $V_\eta^\vee:\RR_{\ge 0}\ra \RR^2_{\ge 0}$ for each $V_\eta$.
The set of points $\{ V_\eta^\vee(1)\}_\eta$ is contained in the segment $\{(l-\alpha,\alpha)|\alpha\in[0,l]\}$ that we identify with $[0,l]$.
We refer to the images of $1$ under $V_\eta^\vee$ as \emph{vertices}
\begin{equation}
\label{vertices-map}
 V_\eta^\vee(1) \in  [0,l] = \{\phi(e_2)\,|\, \phi\in \Hom(\RR^2_{\ge 0},\RR_{\ge 0}), \phi(1,1)=l\}.
\end{equation}
We have just defined a map from the dual intersection graph $\Gamma_C$ of $C$ to $[0,l]$ by mapping the vertex indexed by $\eta$ to $V_\eta^\vee(1)$ and by requiring the map to be linear on edges. (Each edge corresponds to a node $e$ of $C$.) We decree the length of the edge $e$ to be $l_e$. 
By \eqref{eq-basic},
\begin{equation}
\label{vertex-difference}
V_{\eta_2}^\vee(1)-V_{\eta_1}^\vee(1) = l_e w_e
\end{equation}
whenever $e$ is a node between the curve components ${\eta_1}$ and ${\eta_2}$ and the ordering $\eta_1,\eta_2$ is compatible with the orientation of $u_e$ in the sense of \eqref{butterfly}. Consequently, $w_e$ is the scaling factor of the linear map $e \ra [0,l]$ and we take it to be $0$ if $u_e=0$.
The so defined map $h:\Gamma_C\ra [0,l]$ of the metric graph $\Gamma_C$ is a \emph{tropical curve} for which we give a definition below. 
The first relevant property is that, by \eqref{vertices-map} and \eqref{vertex-difference}, $h$ satisfies the \emph{balancing condition} (see \cite[Proposition 1.15]{GS}) which is an equality
\begin{equation}
\label{balancing-cond}
 \sum_{V\in e} \pm u_e= 0,
\end{equation}
for each vertex $V=V_\eta$ of $\Gamma_C$ that corresponds to a component $\overline{\eta}$ of $C_{\bar s}$ that is contracted by $f$. 
The sum is over all nodes $e$ in $\overline{\eta}$, the sign $\pm$ is such that $\pm u_e$ points away from $V_\eta^\vee(1)$.

For an integral monoid $M$, we denote by $M\otimes \RR_{\ge 0}$ the convex hull of $M$ in $M^\gp\otimes_\ZZ\RR$.
Note that \eqref{vertex-difference} induces a partial ordering on the vertices of $\Gamma_C$, i.e., $V_1\le V_2$ if there is an edge $e$ between them and $h(V_1)\le h(V_2)$ holds as points in $[0,l]$. 
The ordering of the vertices of $\Gamma_C$ obtained this way only depends on the minimal face of $Q^\vee_{\op{basic}}\otimes {\RR_{\ge 0}}$ that $\rho$ is contained in. 
In the following, we will always consider the ordering ``$\le$'' obtained from some $\rho$ that lies in the interior of $Q^\vee_{\op{basic}}\otimes {\RR_{\ge 0}}$. 
By continuity,
the vertices of an element $\rho$ in the boundary of $Q^\vee_{\op{basic}}\otimes {\RR_{\ge 0}}$ still satisfy the order induced from an element in the interior.
Hence, the partial ordering ``$\le$'' we will be satisfied by \emph{all} elements of $Q^\vee_{\op{basic}}\otimes {\RR_{\ge 0}}$.

We define
$$\bar Q^\vee_{\op{basic}}:= \{((V_\eta)_\eta,(l_e)_e)\in Q^\vee_{\op{basic}}\,\mid \,l_e=0 \hbox{ whenever }u_e=0\}$$
and $Q_0:=\bigoplus_{e: u_e=0} \NN$ and conclude from close inspection of \eqref{eq-basic} the following Lemma.
\begin{lemma} 
\label{decompose-Qdual}
$Q^\vee_{\op{basic}} =  \bar Q^\vee_{\op{basic}} \oplus Q^\vee_0$
\end{lemma}

By \eqref{decomposeMbar}, there are three possibilities for $P_\eta$, namely $\NN e_1$, $\NN e_2$ or $\NN^2$, depending on whether $\eta$ maps to $X_1\setminus X_2$, $X_2\setminus X_1$ or $D$.
\begin{definition} 
We call a vertex $V=V_\eta$ of $\Gamma_C$ $i$-\emph{rigid} if $P_\eta=\NN e_i$, i.e., $f(\eta)\not\in X_{3-i}$.
\end{definition}

For for $l\ge 0$, let $M_{\Gamma_C,l}$ denote the parameter space of pairs consisting of a tuple of edge lengths $(l_e)_{e}\in(\RR_{\ge 0})^{E(\Gamma_C)}$ and a continuous map $h:\Gamma_C\ra [0,l]$ where each edge $e$ of the graph $\Gamma_C$ is equipped with the metric affine structure of the interval $[0,l_e]$ and $h$ is affine linear on each edge and is
subject to the following constraints.
\begin{enumerate} 
\item[(T1)] 
The scaling factor of the restriction of $h$ to an edge $e$ of $\Gamma_C$ is $w_e$, 
\item[(T2)] the balancing condition \eqref{balancing-cond} holds for vertices that correspond to contracted components, 
\item[(T3)] \eqref{vertex-difference} is satisfied, 
\item[(T4)] $h$ maps the vertices respecting the partial ordering and finally 
\item[(T5)] 1-rigid vertices map to $0$ and 2-rigid vertices map to $l$.
\end{enumerate}
We call such a pair $((l_e)_{e},h)$ a \emph{tropical curve}.
The set $M_{\Gamma_C,l}$ can be identified with a polyhedron in the vector space $\RR\times(\RR^{E(\Gamma_C)})$ by picking any vertex $V_0$ of $\Gamma_C$ and mapping a tropical curve $h$ to the tuple $(h(V_0),(l_e)_e)$ given by the image of $V_0$ and the tuple of edge lengths. 
Let $M_{\Gamma_C}$ be the union $\bigcup_{l\ge 0} M_{\Gamma_C,l}$ which embeds in $\RR\times\RR\times \RR^{E(\Gamma_C)}$ as a convex cone by mapping $h$ to $(l,h(V_0),(l_e)_e)$. In particular, elements in $M_{\Gamma_C}$ can be added, i.e., the sum of tropical curves $h_1:\Gamma_C\ra [0,l_1]$ and $h_2:\Gamma_C\ra [0,l_1]$ is a tropical curve $h:\Gamma_C\ra [0,l_1+l_2]$.

Let $\bar M_{\Gamma_C,l}\subseteq M_{\Gamma_C,l}$ denote the subset of tropical curves with $l_e=0$ whenever $u_e=0$. 
The subset $\bar M_{\Gamma_C,l}$ is a polytope in $\RR\times\RR^{E(\Gamma_C)}$ because it is closed and for each $e$ holds $0\le l_e\le l$. We denote by $\bar M_{\Gamma_C}= \bigcup_{l\ge 0} \bar M_{\Gamma_C,l}$ the subcone of $M_{\Gamma_C}$. 
\begin{lemma} 
\label{trop-moduli-spaces}
\begin{enumerate}
\item $M_{\Gamma_C,l} = \{\rho\in (Q^\vee_{\op{basic}}\otimes \RR_{\ge 0}) \,\mid \, \rho(\one) = l \},$
\item $\bar M_{\Gamma_C,l} = \{ \rho\in (\bar Q^\vee_{\op{basic}}\otimes \RR_{\ge 0}) \,\mid \, \rho(\one) = l \},$
\item $M_{\Gamma_C,0} = Q^\vee_0\otimes \RR_{\ge 0}.$
\end{enumerate}
\end{lemma} 
\begin{proof}
All statements follow from the discussion before, except for the rigidity of vertices which holds because for a 1- or 2-rigid vertex, the map $V_\eta:P_\eta\ra\NN$ is entirely determined by $\one\mapsto l$, so the composition with $\NN^2\sra P_\eta$ now maps $e_1\mapsto l, e_2\mapsto 0$ in the 1-rigid case or the other way round in the 2-rigid case.
Note that (3) follows from \eqref{vertex-difference} because it implies $l_e=0$ whenever $u_e\neq 0$.
\end{proof}

Let $\Gamma$ be the metric graph obtained from $\Gamma_C$ by collapsing all edges with $u_e=0$. 
To be more precise, collapsing means that we inductively identify the vertices of an edge $e$ if $u_e=0$ and we delete the edge in the process, so that every edge $e$ of the resulting graph satisfies $u_e\neq 0$.
\begin{corollary} 
\label{moduli-GammaC}
$\bar M_{\Gamma_C,l}$ is the parameter space of tropical curves $h:\Gamma\ra[0,l]$ satisfying the conditions inherited from $\Gamma_C$.
\end{corollary}

For a toric monoid $Q$, we denote by $Q[1]$ the finite set of primitive generators for the rays, i.e., the primitive elements in the dimension one faces of $Q$.
\begin{definition} 
Given $\rho=((V_\eta)_\eta,(l_e)_e)\in Q^\vee_{\op{basic}}[1]$, we call a node $e$ of $C$ with $l_e\neq 0$ a \emph{splitting node}.
\end{definition}
Recall $\Omega(g, n, \beta)$ from \S\ref{sec-graphs}. 
In the remainder of this section, we are going to define a map 
\begin{equation}
\label{graph-to-curve}
\op{Trop}:\,\left\{ {C/s \ra X/\bk \hbox{ is a basic stable log map over a point }s}\atop{\hbox{together with }\rho\in Q^\vee_{\op{basic}}[1]\hbox{ such that }\rho(\one)\neq 0}\right\}\quad \ra\quad \Omega(g, n, \beta).
\end{equation}

\begin{lemma} 
$\left\{\rho\in Q^\vee_{\op{basic}}[1]\,\mid\,\rho(\one)\neq 0\right\}\quad = \quad\bar Q^\vee_{\op{basic}}[1]$
\end{lemma}
\begin{proof} 
Since $Q^\vee_{\op{basic}}[1]$ is the disjoint union of $\bar Q^\vee_{\op{basic}}[1]$ and $Q^\vee_{0}[1]$, the assertion follows directly from part (3) of Lemma~\ref{trop-moduli-spaces}.
\end{proof}
The lemma implies that $\one$ does not lie in any proper face of $\bar Q_{\op{basic}}$.

\begin{figure}
\includegraphics{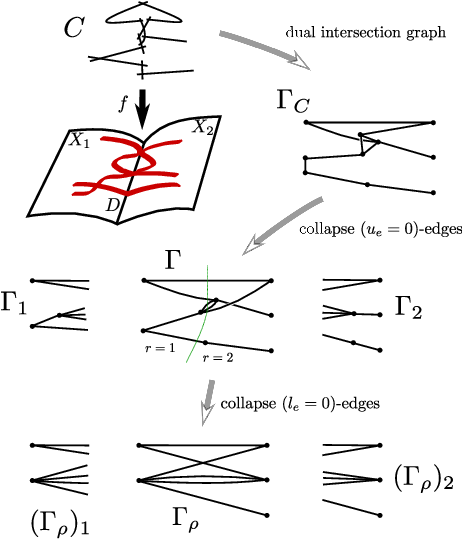}
\caption{An example of a stable log map and its associated graphs for a particular choice of $\rho$.}
\label{graphs-from-curves}
\end{figure}
We now define the map $\op{Trop}$. 
Let therefore $f:C/s \ra X/\bk$ and $\rho\in Q^\vee_{\op{basic}}[1]$ with $l:=\rho(\one)> 0$ be given. Consider the associated tropical curve $h:\Gamma\ra[0,l]$. We will modify $\Gamma$ to a bipartite graph $\Gamma_\rho$, see Figure~\ref{graphs-from-curves} for an example.
\begin{lemma} 
\label{lem-bipartition}
All vertices of $\Gamma$ map to either $0$ or $l$, hence we obtain a map 
$$r:\{\hbox{vertices of }\Gamma\}\ra \{1,2\},\qquad r(V)=\left\{\begin{array}{rl}1&\hbox{ if }V=V_\eta^\vee(1)=0,\\2&\hbox{ if }V=V_\eta^\vee(1)=l.\end{array}\right.$$
\end{lemma}
\begin{proof} 
Recall the definition of $V_\eta^\vee(1)$ from \eqref{vertices-map}.
Assume to the contrary that the set of vertices
$$\Big\{V_\eta^\vee(1)\,\Big|\,\eta\hbox{ is the generic point of a curve component, }V_\eta^\vee(1)\neq 0,l\Big\}$$
is non-empty and let $V_1<...<V_s$ be an enumeration of the set. If $s=1$, set $V_2:=l$. Let $\epsilon>0$ be smaller than $(V_2-V_1)/2$.
We obtain a sum decomposition of vectors with strictly increasing entries 
$$(0,V_1,V_2,...,V_s,l)\ =\ (0,V_1/2-\eps,V_2/2,...,V_s/2,l/2)\ +\ (0,V_1/2+\eps,V_2/2,...,V_s/2,l/2),$$
and the summands on the right are linearly independent.
We can now write the tropical curve $h:\Gamma\ra[0,l]$ as a sum of tropical curves $h_1,h_2:\Gamma\ra[0,l/2]$ as follows. We require for a vertex $V$ of $\Gamma$ that
$h_i(V)=h(V)/2$ unless $h(V)=V_1$ in which case we set $h_1(V)=(V_1-\epsilon)/2$ and $h_2(V)=(V_1+\epsilon)/2$. 
With these prescriptions of where to map the vertives, there is a canonical choice of edge lengths $l_e$ for $h_1,h_2$ so that all defining conditions of a tropical curve are satisfied for $h_1$ and $h_2$. By construction, 
$h_1,h_2$ correspond to elements $\rho_1,\rho_2\in Q^\vee_{\op{basic}}\otimes\RR_{\ge 0}$ that satisfy $\rho_1+\rho_2=\rho$. However, $\rho_1,\rho_2$ are linearly independent and this contradicts the assumption $\rho\in Q^\vee_{\op{basic}}[1]$ because $\rho_1,\rho_2$ span a face of dimension at least $2$ and $\rho$ is contained in its relative interior.
\end{proof}
Equipped with the statement of Lemma~\ref{lem-bipartition}, we collapse all edges in $\Gamma$ that map constantly under $h$ (i.e., those that are not splitting nodes) and obtain a graph $\Gamma_\rho$ that is bipartite by means of the map $r:\{\hbox{vertices of }\Gamma_\rho\}\ra \{1,2\}$ induced from Lemma~\ref{lem-bipartition}.
Each vertex $V$ of $\Gamma_\rho$ is an equivalence class of vertices of the dual intersection graph $\Gamma_C$ of $C$ and thus a vertex $V$ of $\Gamma_\rho$ represents a connected union of curve components that we call $C_V$. 
Note that $C_V$ maps entirely into $X_{r(V)}$.
We decorate $V$ with the genus $g_V=g(C_V)$, curve class $\beta_V=[C_V]$ and $n_V=\{\hbox{markings on }C_V\}$ and then $\Gamma_\rho$ satisfies
 \eqref{genus}, \eqref{class_sum}, \eqref{partition_n}, \eqref{stability} because $\Gamma_C$ satisfies similar conditions.
It remains to verify \eqref{rel_class_sum} in order to have defined the map $\op{Trop}$ in \eqref{graph-to-curve} completely:
\begin{lemma} 
Given $V\in \Gamma_\rho$, we have
\label{eq-tangencies}
$$\beta_V.D =\deg(\shO_{X_{r(V)}}(D)|_{C_V})=\sum_{{e\in E(\Gamma_\rho)}\atop{V\in e}} w_e.$$
\end{lemma}
\begin{proof}
The first equality is clear. In order to prove the second equality, we need to recall the homomorphism $\tau_V:\Gamma(\tilde C_V,g^*\oM_X)\ra\ZZ$ from equation (1.10) in \cite{GS}. Here, $g\colon \tilde C_V\ra C_V\stackrel{f_V}{\ra} X$ is the composition of the normalization $\nu:\tilde C_V\ra C_V$ of an irreducible component $C_V$ of $C$, corresponding to a vertex $V$ of $\Gamma_C$, with the restriction $f_V$ of the stable map $f\colon C\ra X$ to $C_V$. Each section of $g^*\oM_X$ corresponds to an $\shO^\times_{\tilde C_V}$-torsor and the map $\tau_V$ associates to the torsor the degree of the corresponding line bundle.
The description of $\oM_X$ in \eqref{decomposeMbar} leads to a similarly simple description of $\Gamma(D,g^*\oM_X)$, namely 
$$\Gamma(\tilde C_V,g^*\oM_X) = \Gamma(\tilde C_V,g^{-1}\iota_1^{-1}\NN)\oplus \Gamma(\tilde C_V,g^{-1}\iota_2^{-1}\NN) = \NN^{\pi_0(g^{-1}(X_2))}\oplus \NN^{\pi_0(g^{-1}(X_1))},$$
and, by Lemma~\ref{DF-log}, the generators of the two occurences of the monoid $\NN$ in the middle correspond to the torsors $\shL_1,\shL_2$ respectively. 
We can say precisely how the map $\tau_V$ acts on each summand of $\NN$ on the right. 
For a connected compoment of $g^{-1}(X_i)$ that is a single point $x$, the map $\tau_V$ sends the corresponding generator of $\NN$ to the (positive) degree of the Cartier divisor $g^{-1}(X_i)$ at $x$. On the other hand, if $g^{-1}(X_i)$ is all of $\tilde C_V$, then 
$\tau_V$ maps the corresponding generator of $\NN$ to $\deg(g^*\shL_i)$. 
By part (2) of Lemma~\ref{DF-log}, we have $\deg(g^*\shL_i)=-\deg(g^*\shL_{3-i})$ and by part (1) the restriction of $\shL_{3-i}$ to $X_i$ is $\shO_{X_i}(-D)$. Hence,
$$\deg(g^*\shL_i)= -\deg(g^*\shL_{3-i}) = -\deg(g^*\shO_{X_i}(-D)) = (g:\tilde C_V\ra X_i).D.$$
In any event, the sum of the images of the generators of $\NN^{\pi_0(g^{-1}(X_2))}\oplus \NN^{\pi_0(g^{-1}(X_1))}$ under $\tau_V$ is zero and the sum of the images of
$\NN^{\pi_0(g^{-1}(X_i))}$ under $\tau_V$ equals $\deg(g^*\shL_i)$.

If $x\in\tilde C_V$ is a point that maps to a node $e$ of $C$ under the composition $\tilde C_V\ra C_V\hra C$ then $(g^*\oM)_x=P_{e}$ and we have the map $u_e:(g^*\oM)_x\ra\ZZ$ that we naturally extend to a map $\Gamma(\tilde C_V,g^*\oM_X)\ra \ZZ$ by composing with the natural map $\Gamma(\tilde C_V,g^*\oM_X)\ra (g^*\oM)_x$. The general balancing condition as proved in \cite{GS},\,Proposition 1.15 says that 
\begin{equation} \label{general-balance}
\tau_V + \sum_{x} \pm u_e = 0
\end{equation}
where the sum is over precisely those points $x\in\tilde C_V$ that map to nodes of $C$ and the sign $\pm$ is chosen to account for the ordering of the components adjacent to the node in the definition of $u_e$. The sign is $+1$ iff $V$ is the first component in that ordering and if both adjacent components are $V$, i.e., $e$ is a node of $C_V$, then the sum $\sum_{x} \pm u_e$ has the corresponding summand $u_e$ occuring twice with opposite signs, so we can ignore such nodes altogether when forming the sum. 
Recall that $u_e=w_e\tilde u_e$ where $\tilde u_e:\NN^2\ra\ZZ$ is either $(-1,1)$ or $(1,-1)$.
Evaluating \eqref{general-balance} on the generator of $\NN^{\pi_0(g^{-1}(X_i))}=\NN$ for $i$ chosen so that $C_V$ maps into $X_i$ yields 
\begin{equation} \label{general-balance-evaluated}
\beta_V.D + \sum_{x} \pm w_e = 0
\end{equation}
which is already close to the assertion. 
So far we only studied a single component of $C$, however, a single vertex $V'$ of $\Gamma_\rho$ correspond to several vertices $V$ of $\Gamma_C$, namely those that contract to $V'$. The assertion follows from summing up the equation \eqref{general-balance-evaluated} over all $V$ that contract to a specific vertex $V'$ of $\Gamma_\rho$. Necessarily, all associated components $C_V$ map into $X_i$ for $i=r(V')$ and $\beta_{V'}=\sum_V\beta_V$ and evaluating all $\tau_V$ on the generator of
$\NN^{\pi_0(g^{-1}(X_i))}=\NN$ respectively and summing over the $V$ that contract to $V'$ yields $\beta_{V'}.D$ as an intersection evaluated in $X_i$.
Those summands $w_e$ that correspond to a non-splitting edge will appear twice and with opposite sign in the sum and therefore cancel. The contribution from the splitting edges however all carry the same sign (either +1 or -1) because the sum of $\pm w_e\tilde u_e$ over the splitting edges $e$ is a sum of vectors pointing into the interval $[0,l]$ from either the endpoint $0$ or $l$ depending on whether $r(V')=1$ or $r(V')=2$. Evaluating also $\pm \tilde u_e$ on the generator of $\NN^{\pi_0(g^{-1}(X_i))}=\NN$ for each $V$ yields $-1$ and so the assertion follows.
\end{proof} 

\subsection{Generization}
\label{sec-specialization}
We have so far considered a curve over a single point $s$ in this section. 
Let us consider the case where $s$ is in the Zariski closure of another point $\eta$. 
A node of $C_s$ either gets smoothed in $C_\eta$ or it remains a node.
Hence, there is a natural collapsing map of dual intersection graphs $\Gamma_{C_s}\ra \Gamma_{C_\eta}$ and a natural map 
\begin{equation}
\label{specialization-map-Qbasic}
Q^s_{\op{basic}}\ra Q^\eta_{\op{basic}}
\end{equation}
that is a localization composed with modding out the resulting subgroup of invertibles. 
Dually, $(Q^\eta_{\op{basic}})^\vee\subseteq (Q^s_{\op{basic}})^\vee$ is the embedding of a face and hence
$(Q^\eta_{\op{basic}})^\vee[1] \subseteq (Q^s_{\op{basic}})^\vee[1]$. 
Given $\rho\in (Q^\eta_{\op{basic}})^\vee[1]$, the map \eqref{specialization-map-Qbasic} maps $\one$ to $\one$ and commutes with $\rho$, so we get the same $l=\rho(\one)$ for $s$ and $\eta$.
If a node $e$ gets smoothed under generization then $q_e$ (see just after Definition~\ref{def-stablemap}) maps to zero under \eqref{specialization-map-Qbasic}, hence $\rho(q_e)=l_e=0$, so the node $e$ is not a splitting node.
We conclude the following lemmata.
\begin{lemma} 
\label{specialization} If $s\in\overline\eta$ and $\op{Trop}_s$, $\op{Trop}_\eta$ denote the respective maps given in \eqref{graph-to-curve}, then
$\op{Trop}_\eta$ is the composition of the injection $$\left\{\left.\rho \in (Q^\eta_{\op{basic}})^\vee[1]\,\right|\,\rho(\one)\neq 0\right\}\quad \ra\quad \left\{\left.\rho \in(Q^s_{\op{basic}})^\vee[1]\,\right|\,\rho(\one)\neq 0\right\}$$
with $\op{Trop}_s$. 
In particular, for every $\rho\in (Q^\eta_{\op{basic}})^\vee[1]$ with $\rho(\one)\neq 0$, the stable log maps over $s$ and $\eta$ together with $\rho$ respectively give the same tropical curve $\Gamma_\rho\ra[0,l]$.
\end{lemma}
If $\oM$ is a sheaf of monoids on a scheme $S$, we call a subsheaf $\oF\subset\oM$ a \emph{sheaf of facets} if $\oF_x\subset\oM_x$ is a facet for every $x\in S$.
If $M$ is a toric monoid, then its facets are in one-to-one correspondence with the elements $\rho\in M^\vee[1]$ by mapping $\rho$ to $\rho^\perp:=\{m\in M|\rho(m)=0\}$.
By standard toric geometry, if $\rho\in (Q^\eta_{\op{basic}})^\vee[1]$,
the generization map \eqref{specialization-map-Qbasic} sends the facet $F^s_\rho=\rho^\perp$ surjectively onto
the facet $F^\eta_\rho=\rho^\perp \subset Q^\eta_{\op{basic}}$. 
Every other facet of $Q^s_{\op{basic}}$ does not map to a facet under \eqref{specialization-map-Qbasic}.
This analysis implies the following two statements.
\begin{lemma} 
\label{lem-sheaf-facets}
If $C/S\ra X/\bk$ is a basic stable log map, $s\in S$ a point, $\rho\in (\oM_{S,s})^\vee[1]$ and $F^s_\rho=\rho^\perp\subset \oM_{S,s}$ then by the coherence of the log structure on $S$ there is a unique maximal closed subset $W$ of $\Spec\shO_{S,s}$ together with a sheaf of facets $\oF\subseteq \oM_S|_W$ so that $\oF_s=F^s_\rho$. 
\end{lemma} 

\begin{proposition} 
\label{prop-generization}
Let $C/S\ra X/\bk$ be a basic stable log map with $S$ connected and $\rho\in\Gamma(S,\oM_S^\vee)$ with $\rho(\one)\neq 0$ such that $\rho$ maps to an element of $\oM_{S,s}^\vee[1]$ for each $s\in S$.
Then all tropical curves $h:\Gamma_\rho\ra [0,l]$ obtained from $\rho$ at different points $s\in S$ are naturally identified and the corresponding facets $F^s_\rho$ define a sheaf of facets $\oF_S\subset\oM_S$.
\end{proposition}


\section{Splitting stable log maps}\label{splitting}
As in the previous section, consider a stable log map $C/s \ra X/\bk$. Let $\oM_s = Q_{\op{basic}}$ be the associated basic monoid (the dual of $Q^\vee_{\op{basic}}$ in \eqref{eq-basic}). 
We also fix a primitive ray generator $\rho\in Q_{\op{basic}}^\vee[1]$ with $l:=\rho(\one)>0$. 
 The dual intersection graph $\Gamma_C$ of $C$ collapses to $\Gamma$ and then further to $\Gamma_\rho$. The map $r:\Gamma\ra\{1,2\}$ from Lemma~\ref{lem-bipartition} lifts uniquely to $r:\Gamma_C\ra\{1,2\}$ by composition with the collapsing.
Let $\Gamma_i$ denote the possibly disconnected subgraph of $\Gamma$ given by the vertices with $r(V)=i$ and furthermore we include ``half-edges'' at these vertices, one for each edge of a splitting node, see Figure~\ref{graphs-from-curves} for an example.
We similarly define $(\Gamma_\rho)_i$ which is obtained from $\Gamma_i$ by collapsing $(l_e=0)$-edges. 
We also similarly define $(\Gamma_C)_i$. 
The set of vertices of $(\Gamma_C)_i$ inherits the partial order from $\Gamma_C$. 
We call a continuous map $h:(\Gamma_C)_1\ra[0,\infty)$ a tropical curve if it satisfies the analogous conditions (T1) to (T5). 
Here, (T5) is applied only to $1$-rigid vertices. We similarly obtain a notion of tropical curve for maps $h:(\Gamma_C)_2\ra(-\infty,0]$.
Next consider the set
$$Q_1^\vee := \left\{ h:(\Gamma_C)_1\ra[0,\infty) \,\left|\, \begin{array}{c}  h \hbox{ is a tropical curve with $h(V)$ integral for all vertices $V$,}\\ l_e\in\ZZ_{\ge 0}\hbox{ for all compact edges }e,\hbox{ rigid vertices map to }0\end{array}\right.\right\}.$$
We similarly define $Q_2^\vee = \left\{ h:(\Gamma_C)_2\ra(-\infty,0] \,|...\right\}$. Note that $Q_1^\vee, Q_2^\vee$ are monoids.
Since $(\Gamma_C)_1$ decomposes into connected components, we have
\begin{equation}\label{decomposeQ1}
Q_1^\vee=\bigoplus_{r(V)=1} Q_V^\vee
\end{equation}
where the sum is over the vertices of $(\Gamma_\rho)_1$
and $Q_V^\vee$ is the parameter space of tropical curves with domain the component of $(\Gamma_C)_1$ indexed by $V$.
We similarly define $\bar Q_1^\vee := \left\{ h:\Gamma_1\ra[0,\infty)\mid...\right\}$, we have $\bar Q_1^\vee=\bigoplus_{r(V)=1} \bar Q_V^\vee$ and a similar statement for $\bar Q_2^\vee$. Set $Q_i:=(Q_i^\vee)^\vee$ and $\bar Q_i:=(\bar Q_i^\vee)^\vee$.
\begin{lemma} 
\label{split-facet}
The facet $F_\rho:=\rho^\perp\subset Q_{\op{basic}}$ associated to $\rho$ satisfies
$$F_\rho=Q_1\times Q_2.$$
\end{lemma}
\begin{proof} 
In light of Lemma~\ref{decompose-Qdual}, first note that it suffices to prove a similar statement for the facet $\bar F_\rho=\rho^\perp$ of $\bar Q_{\op{basic}}$, the dual of $\bar Q^\vee_{\op{basic}}$. 
Indeed, the duals of the summands of $Q_0^\vee=\oplus_e \NN$ get distributed over $Q_1$ and $Q_2$ depending on whether the edge $e$ contracts to $0$ or $l$ under the tropical curve map $h$ corresponding to $\rho$.

We prove the dual statement, i.e., $\bar F_\rho^\vee = \bar Q^\vee_1\oplus \bar Q^\vee_2.$
Note that $\bar F_\rho^\vee=(\bar Q^\vee_{\op{basic}} + \ZZ\rho) / \ZZ\rho$.
There is a natural homomorphism of monoids 
\begin{equation} \label{project-to-splitting}
\pi:\bar Q^\vee_{\op{basic}}\ra \bar Q^\vee_1\oplus \bar Q^\vee_2
\end{equation}
that maps a tropical curve $h:\Gamma \ra [0,l']$ to the pair $(h_1:\Gamma_1 \ra [0,\infty), h_2:\Gamma_2 \ra (-\infty,0])$ by splitting the curve $h$ at the splitting edges and turning these edges into rays (and translating $l'$ to zero for $h_2$).
We verify that the map is surjective, so we pick a pair $(h_1,h_2)$ on the right hand side.
Take $l_0\in\NN$ larger than the sum of all $l_e$ occurring in $h_1$ and $h_2$. Now translate $h_2$ by $l_0$ to become $\Gamma_2\ra(-\infty,l_0]$.
We can extend this combination of maps of vertices of $\Gamma_1,\Gamma_2$ to a viable tropical curve $h:\Gamma\ra [0,l_0]$ by giving an edge $e$ between vertices $V_1,V_2$ with $r(V_i)=i$ the length $l_e=(h_2(V_2)-h_1(V_1))/w_e$, modifying $l_0$ if needed to ensure that each $l_e$ is integral. 
One verifies that (T1)-(T5) hold, so we verified the surjectivity of $\pi$.
Finally, we need to show that $\pi^{-1}(0)=\NN\rho$. 
A curve that maps to zero under $\pi$ is characterised by the property that all vertices of $h_1,h_2$ are zero (in $[0,\infty)$ and  $(-\infty,0]$ respectively). 
In terms of edge lengths of the original curve, these are either zero if the edge $e$ is not a splitting edge or otherwise $l_ew_e=l'$ for some fixed positive integer $l'$ if the edge is a splitting edge. Such a curve is precisely $\frac{l'}l \rho$ where $l=\rho(\one)$ denotes the length of the interval $[0,l]$ that the tropical curve represented by $\rho$ maps to.
\end{proof}

Say we are given a basic stable log map $C/S\ra X/\bk$ with $S$ connected and also a $\rho\in\Gamma(S,\oM_S^\vee)$ that maps to an element of $\oM_{S,s}^\vee[1]$ for all $s\in S$. 
By Proposition~\ref{prop-generization}, this induces a sheaf of facets $\oF_S\subset \oM_{S}$ that is on stalks given by $F_\rho=\rho^\perp$. 
We obtain a new log structure on $S$ via $\shF_S:=\shM_S\times_{\oM_S}\oF_S$.

Also by Proposition~\ref{prop-generization}, we obtain the same tropical curve $h:\Gamma_\rho\ra [0,l]$ from all points of $S$. 
After replacing $S$ by a finite connected cover if needed, we can order the edges of $\Gamma_\rho$ as $e_1,...,e_r$ and denote this edge-marked curve by $\tilde\Gamma_\rho$. In other words, we mark the splitting nodes $e_i:S\ra C$.
Let $C_i$ be the possibly disconnected union of components of $C$ that are $(l_e\!=\!0)$-edge-contraction-equivalent to vertices $V$ of $\Gamma_\rho$ with $r(V)=i$. 
Since $C_1$ and $C_2$ intersect in the splitting nodes, we have a cocartesian (alias pushout) diagram
\begin{equation}\label{glue-diag} 
\begin{CD} 
        \uS^r @>{(e^1_1,...,e^1_r)}>> \underline C_1 \\
        @V{(e^2_1,...,e^2_r)}VV @VVV \\
        \underline C_2 @>>> \underline C.
\end{CD}
\end{equation}
Recall that $C_i$ maps into $X_i$ under $f:C\ra X$. 
In the following, we set $i=1$. By symmetry, the case $i=2$ works analogously.
As said in Remark~\ref{rem-log-Xi}, $X_1$ carries the divisorial log structure by $D\subset X_1$ and there is a natural log morphism
$$(\uX_1,\shM_X|_{\uX_1})\ra X_1$$
via the injection $\shM_{X_1}\hra \shM_X|_{\uX_1}$.
We may restrict $f:C\ra X$ (as a log map) to $C_1$, i.e., $\shM_{C_1}=\shM_C|_{C_1}$, and compose with the above map to obtain a map $C_1\ra X_1$.
We will find a natural sub-log-structure $\shF_{C_1}\subset \shM_{C_1}$ giving a commutative diagram
\begin{equation}
\label{C1-square} 
\begin{CD} (C_1,\shF_{C_1}) @>f_1>> X_1 \\
        @V{\pi}VV @VVV \\
        (\uS,\shF_S) @>>> \ubk. 
\end{CD}
\end{equation}
The left vertical map $\underline C_1\ra \uS$ has natural sections $e^1_1,...,e^1_r$ in addition to the usual markings given by the markings of the splitting nodes of $C$. 
Including the additional markings, the diagram of underlying schemes in \eqref{C1-square} constitutes a stable map because $f$ is a stable map. 
To find $\shF_{C_1}$, away from nodes and marked points on $C_1$, we simply take the pullback $\pi^*\shF_S$ (i.e., make $\pi$ strict there).
Furthermore, it suffices to give $\oF_{C_1}\subset \oM_{C_1}$ and then set $\shF_{C_1}:=\shM_{C_1}\times_{\oM_{C_1}} \oF_{C_1}$. 
Generic strictness reduces this to a local problem, looking at markings and nodes. At an ordinary marked point $x$, we have $\oM_{C_1,x}=\oM_{C,x}=\oM_{S,\pi(x)}\oplus\NN$ and we pick the substalk $\oF_{C_1,x}:=\oF_{S,\pi(x)}\oplus\NN$.
At a node $x$ in $C_1$, so not a splitting node, we set $\oF_{C_1,x} := \oF_{S,x}\oplus_\NN \NN^2 \subset \oM_{S,x}\oplus_\NN \NN^2 = \oM_{C,x}$ and this works because the map $\NN\ra \oM_{S,x}, 1\mapsto q_e$ factors through $\oF_{S,x}$ (indeed it maps to $\rho^\perp$ because $\rho(q_e)=l_e=0$ for all non-splitting nodes). 
Finally, for $x=e^1_j$ a splitting node, we take for $\oF_{C_1,x}$ the submonoid $\oF_{S,x}\oplus\NN\subset \oM_{S,x}\oplus_\NN \NN^2$ where the $\NN$-summand embeds in the second copy (the one that corresponds to $i=2$) on the right.
We thus produced \eqref{C1-square}.

Note that there is a decomposition in connected components $C_1=\coprod_{{V\in\Gamma_\rho}\atop{r(V)=1}} C_V$. 
\begin{proposition} 
\label{prop-split}
Given a basic stable log map $C/S\ra X/\bk$ together with $\rho\in\Gamma(S,\oM_S)$ that maps to an element of $\oM_{S,s}[1]$ for all $s\in S$ and an ordering of the edges of the resulting tropical curve $\Gamma_\rho$,
\begin{enumerate}
\item the diagram \eqref{C1-square} obtained from this input data constitutes a stable log map with contact order data given by the weights of the unbounded edges of $(\Gamma_\rho)_1$. Here $C_1$ is potentially disconnected and 
\item the collection of inclusions $Q_1\subseteq \oF_{S,s}$ for all $s\in S$ given via Lemma~\ref{split-facet} constitutes a subsheaf $\oQ_1$ of monoids of $\oF_S$ and the fibre product
$\shM^1_S:=\shF_S\times_{\oF_S} \oQ_1$ is the basic log structure for the diagram \eqref{C1-square}. 
Similarly, the decomposition $\eqref{decomposeQ1}$ yields subsheaves $\oQ_V\subset \oF_S$ that give the basic log structure $\shM^V_S= \shF_S\times_{\oF_S} \oQ_V$ of the connected components of $C_1$.
Furthermore, the map $\shM^1_S\ra \shF_S$ (respectively $\shM^V_S\ra \shF_S$) realizes \eqref{C1-square} (respectively the $V$-component of it) as the pullback from this basic log structure.
\end{enumerate}
\end{proposition}
\begin{proof} 
The smoothness of $\pi$ follows from the construction of $\shF_{C_1}$ as locally it has precisely the shape as in the classification of log smooth curves \cite[\S1.8]{F.Kato}, \cite[Theorem 1.1]{GS}.
For (1), it remains to study the contact orders. The definition was given just before Definition~\ref{def-beta}.
At a splitting node $e$ in $f:C\ra X$, we identify the map $P_e\ra Q\oplus_{\NN} \NN^2$ in \eqref{butterfly} with the map of stalks of the characteristics at the node $\varphi:\NN^2\ra \oM_{S,s}\oplus_\NN \NN^2$. The part of $\varphi$ that maps to the second summand is the map $\NN^2\ra\NN^2$ that is given by multiplication by $w_e$ which follows from the definition of $u_e$. 
On the other hand, by the preceding construction of $\shF_{C_1,x}$ at a splitting node $x$ as the subsheaf $\oF_{S,x}\oplus\NN\subset \oM_{S,x}\oplus_\NN \NN^2$, restricting $\varphi$ to the second $\NN$-summand yields $\varphi_2\colon \NN\ra \shF_{S,s}\oplus\NN$ and the composition with the projection to the second $\NN$-summand is multiplication by $w_e$, so the weight of the edge $e$ of $(\Gamma_\rho)_1$ gives the contact order as claimed.

For (2), the existence of the sheaf $\oQ_1$ is Lemma~\ref{lem-sheaf-facets}. Note that the labelling of edges of $\Gamma_\rho$ together with the map $r$ makes vertices uniquely identify-able as each vertex is adjacent to at least one edge, so we don't need to additionally enumerate vertices and then consequently \eqref{decomposeQ1} gives sheaves $\oQ_V$ as claimed. That the $Q_V$ are the basic monoids (and then consequently $Q_1$ is also) follows directly from Definition~\ref{def-basic} and Equation \eqref{eq-basic}.
Finally, the statement that the inclusion $\shM_S^V, \shM_S^1\subset \shF_S$ gives the pullback from the basic log structure can be checked directly. Indeed, $\shM_{C_1}=\pi^*\shF_{S}\oplus_{\pi^*\shM^1_{S}}\shM^1_{C_1}$, where the definition of $\shM^1_{C_1}$ is as that of $\shM_{C_1}$ above, only with $\shM^1_{S}$ in place of $\shF_{S}$ everywhere. Similarly, one defines $\shM^V_{C_V}$ and has then $\shM_{C_1}|_{C_V}=\pi|_{C_V}^*\shF_{S}\oplus_{\pi|_{C_V}^*\shM^V_{S}}\shM^V_{C_1}$ as desired.
\end{proof}

A similar version of Proposition~\ref{prop-split} holds for $C_2$ in place of $C_1$, so we finished the splitting procedure that turns a basic stable log map $f:C/S\ra X/\bk$ into a pair of basic stable log maps $f_1:C_1/S\ra X_1/\ubk$ and $f_2:C_2/S\ra X_2/\ubk$ and then we can split further into $C_V$ over vertices $V$ of $\Gamma_\rho$ corresponding to components of $C_1,C_2$. We finished constructing the map $\phi_\tGamma$ in \eqref{diag-moduli-spaces}.

Note that, by construction, there is a map from the original stable map log structure to the split one in \eqref{C1-square}, i.e., we have a commutative diagram
\begin{equation} 
\vcenter{
\label{split-commutes}
\xymatrix{
(\underline C_V,\shM_C|_{\underline C_V})/S \ar[d]\ar[r] & (\uX_{r(V)},\shM_X|_{\uX_{r(V)}})/\bk \ar[d]\\
C_V/(\uS,\shM^V_{S}) \ar[r] & X_{r(V)}/\ubk.
}}
\end{equation}


\section{Gluing stable log maps}\label{gluing}
The purpose of this section is to reverse the process of the last section. 
We assume to be given $\tilde\Gamma\in \tOmega (g, n, \beta)$ and an object in $\bigodot_V \mathscr{M}_V$, see \eqref{diag-moduli-spaces}. I.e., we have two basic stable log maps $f_1:C_1/S\ra X_1/\ubk$ and $f_2:C_2/S\ra X_2/\ubk$ with contact order data $\tilde\Gamma_1$ and $\tilde\Gamma_2$ respectively and the underlying curves with matching contact orders, i.e., $w_{e^1_i}=w_{e^2_i}$ for $e^j_i\in E(\tilde\Gamma_j)$ the $i$th edge for $j=1,2$ and also
$f_1(e^1_i)=f_2(e^2_i)$ for each $i$, so we have the diagram \eqref{glue-diag}.
For a point $s\in S$, denote by $C_{1,s}, C_{2,s}$ the curves above $s$. 
We obtain $Q_i:=\oM^i_{S,s}$ and the interpretation of its dual $Q_i^\vee$ as a parameter space of tropical curves 
$h_1:\Gamma_{C_{1,s}}\ra [0,\infty)$ and $h_2:\Gamma_{C_{2,s}}\ra (-\infty,0]$ given in \S\ref{splitting} respectively.
Plugging $\Gamma_{C_{1,s}}$ and $\Gamma_{C_{2,s}}$ together by gluing half-edges to compact edges along matching $e^1_i\leftrightarrow e^2_i$ yields $\tGamma_C$. We give the resulting new compact edges the weights $w_{e_i}=w_{e^1_i}=w_{e^2_i}$.
The natural map $r:\{\hbox{vertices of }\tGamma_C\}\ra \{1,2\}$ is given by whether a component of the curve is in $C_1$ or $C_2$.
We hence obtain a graph $\tGamma_C$ fully decorated with $w_e,\beta_V,n_V,g_V$.
Collapsing $\tGamma_C$ to a bipartite graph $\tGamma_\rho$ using $r$ and inferring the decorations on $\tGamma_\rho$ from $\tGamma_C$, we find that $\tGamma_\rho$ is an element of $\tOmega (g, n, \beta)$ and in fact $\tGamma_\rho=\tilde\Gamma$. We abuse notation when writing $\tGamma_\rho$ at this point because we have not yet defined $\rho$ that yields this graph.

Our next step is to define the monoid $Q^\vee_{\op{basic},s}$ together with an element $\rho$ so that $\tGamma_\rho$ is the graph associated to $\rho$.
As in the proof of surjectivity of \eqref{project-to-splitting}, we can lift any pair of tropical curves $h_1,h_2$ to a tropical curve $h:\tGamma_C\ra[0,l]$ for some $l\gg 0$.
We define $Q^\vee_{\op{basic},s}$ to be the parameter space of integral tropical curves $h:\tGamma_C\ra[0,l]$ with varying $l\ge 0$ and with the constraints
(T1) to (T5). Here, \emph{integral} simply means that $l$, the $V^\vee_\eta$ and the $l_e$ are all integral.
We define $q_e,\one\in (Q_{\op{basic},s}^\vee)^\vee=:Q_{\op{basic},s}$ respectively as the maps $Q^\vee_{\op{basic},s}\ra\NN$ given by $((V_\eta)_\eta,(l_{e'})_{e'})\mapsto l_{e'}$,  $(h\colon \Gamma_{C_s}\ra[0,l])\mapsto l$.
The monoid $Q_{\op{basic},s}^\vee$ contains a particular element $\rho$ that is given by the tropical curve $h:\Gamma_{C_s}\ra[0,l]$ where
\begin{equation}
\label{lcm}
l=\op{lcm}(w_e:e\hbox{ \small is a splitting node})
\end{equation}
and the $(r=1)$-vertices of $\Gamma_{C_s}$ map to $0$ and the $(r=2)$-vertices map to $l$.
With the same reasoning as in Lemma~\ref{split-facet}, we find that $\rho$ is contained in $Q_{\op{basic},s}^\vee[1]$ and the associated facet $F_\rho=\rho^\perp$ of $Q_{\op{basic},s}$
takes the form $F_\rho=Q_1\times Q_2$ where $Q_i^\vee$ is the parameter space of integral tropical curves that map $\Gamma_{C_{i,s}}$ to a ray as at the beginning of \S\ref{splitting}.
Under the construction in \S\ref{section-graphs-to-curves}, i.e., collapsing $(u_e=0)$- and $(l_e=0)$-edges, it is not hard to see that the tropical curve given by $\rho$ yields precisely the decorated bipartite graph $\tGamma_{\rho}$ that we produced from plugging together $\Gamma_{C_{1,s}}$ and $\Gamma_{C_{2,s}}$ in the above paragraph, except we forgot the ordering of the edges.

By construction and Proposition~\ref{prop-generization}, $\tGamma_{\rho}$ is independent of $s\in S$ and compatible with generization, meaning that for $\eta\in S$ with $s\in\bar\eta$, we have a collapsings $\Gamma_s\ra\Gamma_\eta\ra\tGamma_\rho$.

As the next step, we want to construct a diagram
\begin{equation}
\label{Ghost-square} 
\begin{CD} \oM_{C_s}  @<f^*<< \NN^2 \\
        @A{\pi^*}AA @AA{1\mapsto (1,1)}A \\
        Q_{\op{basic},s} @<{\one\mapsfrom 1}<< \NN
\end{CD}
\end{equation} 
of sheaves of monoids on $\underline C_s$ where all except the top left one are constant sheaves.
We are going to define $\oM_{C_s}$ as a subsheaf of 
$$\oM_{C_s}^{\op{pre}}:=\left(\bigoplus_{V\in \Gamma_{C_s}} i_{V,*}Q_{\op{basic},s}\right)\oplus\left(\bigoplus_{j\in n} \sigma_{j,*}\NN\right)$$
where $i_V:C_V\ra C_s$ is the inclusion of a component. The projection of the image of $f^*$ and $\pi^*$ to the second summand $\left(\bigoplus_{j\in n} \sigma_{j,*}\NN\right)$ will be trivial.
Away from the nodes, we set $\oM_{C_s}=\oM_{C_s}^{\op{pre}}$ and at a node $e$ with adjacent components $V_1,V_2$, the stalk of $\oM_{C_s}$ is defined by requiring that its projection to $i_{V_1,*}Q_{\op{basic},s}\times i_{V_2,*}Q_{\op{basic},s}= Q_{\op{basic},s}\times Q_{\op{basic},s}$ agrees with
$$\{(a,b)\in Q_{\op{basic},s}\times Q_{\op{basic},s}\,\mid\, b=kq_e+a\hbox{ for some }k\in\ZZ \}.$$
By the universal property of the pushout, the latter is canonically isomorphic to 
$$Q_{\op{basic},s}\oplus_{q_e\mapsfrom 1,\NN,1\mapsto (1,1)} \NN^2$$ 
and so we naturally obtain the commutative square \eqref{Ghost-square} for the stalk at each node $e$.
We globalize the map $f^*$ by taking it to be $(V_\eta)_\eta$ (see \eqref{tropical-sequence-0}).
The map $\pi^*$ in \eqref{Ghost-square} globalizes by mapping diagonally into the first summand of $\oM_{C_s}^{\op{pre}}$.

\begin{lemma} \label{lemma-ghost-glued}
The map $f^*$ factors through $f^*\oM_X$ and the diagram \eqref{Ghost-square} is well-defined and commutes.
\end{lemma}
\begin{proof} 
In view of \eqref{decomposeMbar}, for the first claim, we need to show that $V_\eta$ is trivial on the $i$th summand of $\NN\oplus\NN$ whenever $f$ maps the generic point of a component $\eta$ away from $X_i$. 
Mapping $\eta$ away from $X_i$ means $V_\eta$ is $(3-i)$-rigid and by definition the integral tropical curves parametrized by $Q^\vee_{\op{basic},s}$ satisfy the rigidity constraint, so $f^*$ factors through $f^*\oM_X$ as claimed.
The sheaf $\oM_{C_s}$ is well-defined.
That $f^*$ maps into $\oM_{C_s}$ follows from \eqref{node-relation}: indeed, if $V_1,V_2$ are connected by an edge $e$ then
$h(V_1)-h(V_2)=w_el_e$ holds for every integral tropical curve $h\in Q^\vee_{\op{basic},s}$ and $q_e:Q_{\op{basic},s}^\vee\ra\NN$ is the map that returns $l_e$, so $V_1-V_2$ is an integral multiple of $q_e$ as required.
Finally, we check commutativity of the diagram \eqref{Ghost-square} at stalks. At a node, the commutativity follows by the construction of the diagram as a pushout.
At a stalk of $C_s$ which is \emph{not} a node, the composition of $\pi^*$ with the projection to the first summand of $\oM_{C_s}^{\op{pre}}$ is an isomorphism and the dual of the diagram is commutative by the equality $l= \rho( \one ) = \rho(V_\eta(1, 1))$ that holds for every vertex $V_\eta$, see \eqref{vertices-map} and the line after the equation.
\end{proof}

The remainder of this section is about lifting the diagram \eqref{Ghost-square} to actual maps of log structures for a basic stable log map $C/S\ra X/\bk$.
First note that taking $\shF_S=\shM^1_S\oplus_{\shO_S^\times}\shM^2_S$ as a log structure on $S$ and on $C_1,C_2$ the pullbacks $\shM_{C_i}\oplus_{\pi^*\shM^i_{S}}\pi^*\shF_S$, we obtain the diagram \eqref{C1-square} for $i=1,2$.

Since $\underline C/\uS$ is a stable curve, as such it receives a basic log structure from $\mathscr{C}_{g, n}\ra \mathscr{M}_{g, n}$, 
the Artin stack of prestable curves $\mathscr{M}_{g, n}$ with its universal curve $\mathscr{C}_{g, n}$, cf.~\cite[Appendix A]{GS}, \cite[p. 227ff.]{F.Kato}. 
We denote this log structure by $\shM^{C/S}_C$ on $C$ and $\shM^{C/S}_S$ on $S$ and have the induced map
\begin{equation} \label{basic-curve-structure}
\pi^*\shM^{C/S}_S\ra \shM^{C/S}_C.
\end{equation}
For a point $s\in S$ and $C_s$ the fibre over it, we have $\oM^{C/S}_{S,s}=\NN^{E(\Gamma_{C_s})}$ and this is compatible with \eqref{eq-basic} (by having $P_\eta=0$ for all $\eta$).

We arrive at the following maps of sheaves on $S$
\begin{equation} 
\label{generation-by-actual-log-structures} 
\begin{CD} \shM^{C/S}_{S}  @>>> \oM_{S} @<{\one\mapsfrom 1}<< \NN\\
        && @AAA \\
        &&\shF_S
\end{CD}
\end{equation} 
where the top left map sends a generator of the $\NN$-copy indexed by a node $e$ to $q_e$.
\begin{lemma} 
\label{images-generate-MS}
The images of the left and bottom map going into $\oM_{S}$ in \eqref{generation-by-actual-log-structures} generate $\oM_{S}^\gp$.
\end{lemma}
\begin{proof} 
The image contains $\oF_S$ which is co-rank one. We have $\oF_S=\rho^\perp$ and $\rho$ is primitive, so it suffices that we can find an element $q$ in the linear combination of the images that has $\rho(q)=1$.
We claim such an element can be obtained as a linear combination of the $q_e$ which will be clear once we prove
$$\gcd\{ l_e\,\mid\,e\in E(\Gamma_\rho) \}=1$$
since $l_e=\rho(q_e)$. Assume $k|l_e$ for all $e$. Since $w_el_e=l$, we find $k|l$ and $k>1$ would contradict primitivity of $\rho$ since then $\frac1k \rho$ would be integral, so indeed $\gcd=1$ and we are done.
\end{proof}
Our next goal is to lift $\oM_{S}$ to a log structure $\shM_{S}$.
Note that $C_1/S$ and $C_2/S$ are stable curves, so they induce maps $\shM^{C_i/S}_S\ra\shM^i_S$ that we sum to have maps
\begin{equation} 
\label{complement square}
 \shM^{C/S}_{S} \leftarrow \shM^{C_1/S}_S\oplus_{\shO_S^\times} \shM^{C_2/S}_S \rightarrow \shF_S
\end{equation} 
that fit in to fill the empty bottom left corner of \eqref{generation-by-actual-log-structures} giving a commutative square with the maps to $\oM_S$.
We let $\widehat\shM_S$ be the pushout of \eqref{complement square}.
Since all terms in \eqref{complement square} are log structures, it is not hard to see that 
$\widehat\shM_S$ with the natural induced map to $\shO_X$ is also a log structure.
Note also that $\overline{\widehat\shM_S} = \NN^r\oplus\oF_S$ because every stalk $\oM^{C/S}_{S,x}$  of $\oM^{C/S}_{S}$ decomposes as $\oM^{C/S}_{S,x}=\NN^r\oplus \NN^s$ for some $s$ and the map from
$\overline{\shM^{C_1/S}_{S,x}\oplus_{\shO_{S,x}^\times} \shM^{C_2/S}_{S,x}}=\NN^s$ to $\oM^{C/S}_{S,x}$ is the injection $\{0\}\times\NN^s \hra\NN^r\oplus \NN^s$.

We use $\widehat\shM_C:=\pi^*\widehat\shM_S\oplus_{\pi^*\shM^{C/S}_S} \shM^{C/S}_C$ and so the map 
$$ \pi^*\widehat\shM_S\ra \widehat\shM_C$$
 makes $\pi$ log-smooth because it is just the pullback of \eqref{basic-curve-structure}.

However $\widehat\shM_S$ is too large for what we want and the remainder of this section is about producing $\shM_S$ as a suitable quotient of $\widehat\shM_S$.
Note that $\overline{\widehat\shM_S}\ra \oM_S$ is surjective by Lemma~\ref{images-generate-MS} but not an isomorphism if $\tGamma_\rho$ has more than one edge. 
This is because ${\oM^\vee_S}$ parametrizes integral tropical curves with a map to an interval which requires a relation between the edge lengths, see \eqref{vertex-difference}.
This condition is absent in $(\overline{\widehat\shM_S})^\vee$, indeed 
\begin{equation}
\label{hatMS}
\overline{\widehat\shM_S}= \oM^1_S\oplus \oM^2_S\oplus \NN^r = \oF_S\oplus \NN^r
\end{equation}
where $r$ is the number of edges of $\tGamma_\rho$.
We are going to define a global section of $\overline{\widehat\shM_S}$ as a sum
\begin{equation}
\label{split-1-in-three}
\one_e:= V_1 + w_eq_e + V_2
\end{equation}
where $q_e$ is the generator of the $\NN$-summand in \eqref{hatMS} that corresponds to the node $e$. 
This is by slight abuse of notation as the projection of $q_e$ to $\oM_S$ also has this name.
Furthermore, for $s\in S$,  $V_i\in {\oM_{S,s}^i}$ is defined by how it pairs with a tropical curve $h:\Gamma_{C_s^1}\ra[0,\infty)$ or $h:\Gamma_{C_s^2}\ra(-\infty,0]$ parametrized by $(\oM_{S,s}^1)^\vee,(\oM_{S,s}^2)^\vee$ respectively via Lemma~\ref{trop-moduli-spaces}.
We set $V_i:(\oM_{S,s}^i)^\vee\ra\NN$ to be the distance from $0$ of the vertex $V_i$ of $e$.
Note that under the projection $\overline{\widehat\shM_S}\ra\oM_S$ each $\one_e$ maps to $\one$. 
Indeed, it becomes the operator that associates to a tropical curve $h:\Gamma_{C}\ra[0,l]$ the length $l$ since $h(V_2)-h(V_1)=w_el_e$ by \eqref{vertex-difference}, see also Lemma~\ref{lem-identity-one}.

\begin{lemma} 
\label{sequence-descend-to-MS}
For $E(\tGamma_\rho)=\{e_1,...,e_r\}$, the lattice $K:=\ZZ(\one_{e_2}-\one_{e_1})\oplus...\oplus  \ZZ(\one_{e_r}-\one_{e_1})$ injects in $\Gamma(S,\overline{\widehat\shM_S}^\gp)$, let $K^\sat$ denote its saturation.
We have a split exact sequence
$$0 \ra K^\sat \ra \overline{\widehat\shM_S}^\gp\ra \oM_S^\gp \ra 0.$$
\end{lemma}
\begin{proof} 
That $K^\sat$ injects in the middle term is clear and also that it lies in the kernel to the right by what we just said about all $\one_e$ mapping to $\one$ and because $\oM_S^\gp$ is torsion-free. 
Surjectivity on the right is Lemma~\ref{images-generate-MS}. By checking ranks, it is also straightforward to see that the sequence is exact over $\QQ$ which completes the proof up to finding a splitting of the exact sequence. 
Indeed, the proof of Lemma~\ref{images-generate-MS} provide an element $q$ as a linear combination of $q_e$ and we may interpret this linear combination in $\overline{\widehat\shM_S}^\gp$ thus together with $\oF^\gp_S$ producing an injection 
$\overline{\shM_S}^\gp \ra \overline{\widehat\shM_S}^\gp$ that is an inverse to the reversely directed surjection.
\end{proof}
An example where $K\neq K^\sat$ is given by the situation of two edges with the same vertices but weights not coprime.
Define $\shL_{\one_e}$ to be the $\shO^\times_S$-torsor that is the inverse image of $\one_{e}$ in $\widehat\shM_S$.

\begin{lemma} \label{lemma-section-exists}
$\shL_{\one_e}\cong\shO_S^\times$ for all edges $e$ of $\tGamma_\rho$.
\end{lemma}

\begin{proof} 
Let $\shL_{V_1},\shL_{V_2}, \shL_{q_e}$ be the $\shO_S^\times$-torsors that are the inverse images of $V_1,V_2,q_e$ under ${\widehat\shM_S}\ra\overline{\widehat\shM_S}$. 
By \eqref{split-1-in-three}, we have $\shL_{V_1}\otimes\shL_{q_e}^{\otimes w_e}\otimes\shL_{V_2}\cong\shL_{\one_e}$ and want to show this is trivial.\\[4mm]
\begin{minipage}{0.65\textwidth} 
Since $e$ is a node over all points of $S$, we have a section $e:S\ra C$ and sections $e^j:S\ra C_j$ and by \cite[\S 2-\hbox{Global construction}]{F.Kato}, we find $\shL_{q_e}=\shL_{e^1}\otimes \shL_{e^2}$ where $\shL_{e^1},\shL_{e^2}$ is the $\shO^\times_S$-torsor given by the conormal bundle of the marked point $e^1,e^2$ in $C_1,C_2$ respectively.
\end{minipage}\qquad
\begin{minipage}{0.3\textwidth}
\[ \xymatrix{
C_j\ar^{f_j}[r]\ar^{\pi_j}[d] & X_j\\
S\ar@/^/^{e^j}[u] \ar_{\ev_{e^j}}[ur]
} \]
\end{minipage}\\[4mm]
The characteristic $\oM_{X_j}$ is globally generated by the generator $\one_j\in \NN$ that maps to $\oM_{C_j}$. 
The associated torsor, the inverse image in $\shM_{C_j}$, we call $\shL_{\one_j}$. The torsor $\shL_{\one_j}$ is isomorphic to the torsor of the line bundle $f_j^*\shO_{X_j}(-D)$ because the torsor of the $\NN$-generator on $X_j$ is the torsor of $\shO_{X_j}(-D)$ by Lemma~\ref{DF-log} and every map of torsors is an isomorphism.
Next note that $V_j\in \Gamma(S,\oM_S^j)$, i.e., both $V_1,V_2$ lie in the facet $\oF_S$ of $\oM_S$.
We have $\oM_{C_j}|_{e^j}=(\pi_j^*\oM^j_S)|_{e^j}\oplus\NN$ and $\one_j=(V_j,w_e)$ in this, hence
$$((\pi_j^*\shL_{V_j})|_{e^j})\otimes \shL_{e^j}^{w_e}=\shL_{\one_j}|_{e^j}.$$
Now $(e^j)^*\pi_j^*=\id^*_S$ and $(e^j)^*f_j^*=\ev^*_{e^j}$, hence 
$\shL_{V_j}\otimes \shL_{e^j}^{w_e}$ is isomorphic to the torsor of $\ev^*_{e^j}\shO_{X_j}(-D)$.
Now use that on $X$ we have $\one_1+\one_2=\one$ and to $\one$ is associated the trivial torsor since $\NN\ra\shM_X$ is a global section. 
This is just saying $\shO_{X_1}(-D)|_D$ is dual to $\shO_{X_2}(-D)|_D$. 
Putting it all together yields
$$ \shL_{V_1}\otimes \shL_{q_e}^{\otimes w_e}\otimes \shL_{V_2} \ \cong  \ 
\shL_{e^1}^{-w_e}\otimes \ev^*_{e^1}\shO^\times_{X_1}(-D)\otimes (\shL_{e^1}\otimes \shL_{e^2})^{w_e} \otimes \shL_{e^2}^{-w_e}\otimes \ev^*_{e^2}\shO^\times_{X_2}(-D)\  \cong\  \shO^\times_S.$$
\end{proof}
A consequence of Lemma~\ref{lemma-section-exists} is that the inverse image of every element of $K$ in $\widehat\shM^\gp_S$ is a trivial torsor and thus has sections. 
The next step is to produce a section $s_{\one_e}\in\Gamma(S,\shL_{\one_e})$ that is in fact uniquely determined by filling the dashed arrow in the diagram
\begin{equation} 
\vcenter{
\label{fill-dash-arrow-diagram}
\xymatrix{
 (\widehat\shM_{C})_\eta   & (f^*\shM_{X})_\eta\ar[l] \\
(\pi^*\widehat\shM_{S})_\eta \ar[u]   & \NN \ar[u]\ar@{-->}[l] 
} }
\end{equation}
by means of $1\mapsto s_{\one_e}$ in order to make it commutative at stalks at points $\eta$ in the image of the section $\uS\ra\uC$ that marks the node $e$.
Once this is done, we will take a quotient of $\widehat\shM_{S}$ that identifies all these sections, so that we get a map from $\NN$ into the quotient that is defined compatibly for all nodes.

Let $e$ be a gluing node (alias edge of $\tilde\Gamma_\rho$) of a curve $f:\uC/\uS\ra \uX$ glued from $C_1,C_2$ as in \eqref{glue-diag}. Let $V_1,V_2$ be the adjacent vertices of $\Gamma_C$ with $r(V_i)=i$. 
Let $\eta$ be a point in the node locus of $e$, then $\eta$ necessarily maps to $D$ under $f$, so $P_e:=\oM_{X,f(\eta)}=\NN^2$. 
The top horizontal arrow in \eqref{Ghost-square} at $e$ is the map $f_e:P_e\ra \widehat Q\oplus_{\NN} \NN^2$ in \eqref{butterfly} for $\widehat Q:=\overline{\widehat{M}}_{S,\pi(\eta)}$ and this is given by $f_e:\NN^2\stackrel{(V_{\eta_1},V_{\eta_2})}{\lra} \widehat Q\times \widehat Q$, see Lemma~\ref{lemma-ghost-glued}.
We can be more explicit by using \eqref{split-1-in-three} denoting as before $V_i:Q_i^\vee\ra\NN$ the map sending an integral tropical curve to the distance of $V_i$ from the origin. We find
\begin{equation} 
\label{eq-node-map-f-mbar}
f_e:\NN^2\ra \widehat Q\oplus_{\NN} \NN^2,\qquad (\lambda_1,\lambda_2)\mapsto  (\lambda_1V_1+\lambda_2V_2,(w_e\lambda_1,w_e\lambda_2))
\end{equation}
and indeed $f_e(1,1)=(V_1+V_2,(w_e,w_e))\sim (V_1+w_eq_e+V_2,(0,0))=\one_e$ by \eqref{split-1-in-three}.
As part of the datum of $X\ra\bk$, namely the map on log structures $\NN\ra\shM_X$, we are given as the image of $1$ under this, a lift of $(1,1)$ in $\shM_X$ and we look at its localization in $(\shM_X)_\eta$. 
Locally at $f(\eta)$, we can choose a chart of the log structure of $X$ given by 
\begin{equation} \label{eq-chosen-chart-X}
\NN^2\ra \shM_{X,f(\eta)}\ra \shO_{X,f(\eta)}\cong (B[z_1,z_2]/(z_1,z_2))_{(z_1,z_2)}, \quad e_i\mapsto z_i
\end{equation}
for $B=\shO_{D,f(\eta)}$ and $z_i$ a local equation of $D$ in $X_i$ at $f(\eta)$. 
We may assume this chart is compatible with the chart on $\bk$, i.e., $(1,1)$ maps to the given section of $\shM_X$ that comes from the chart of $\bk$.
Using \eqref{eq-chosen-chart-X} and the fact that we are given basic stable log maps $C_i\ra X_i$, we obtain maps
\begin{equation} 
\label{eq-Ci-to-be-glued}
\NN e_i\ra (\shM_{X}|_{X_i})_{f(\eta)} \ra \shM_{X_i,{f(\eta)}}\ra \shM_{C_i,\eta}
\end{equation}
for $i=1,2$ whose composition with $\shM_{C_i,\eta}\ra \oM_{C_i,\eta}=\shM^i_{S,\eta}\oplus\NN$ sends $e_i$ to $(V_i,w_e)$.
In particular, by \eqref{eq-node-map-f-mbar}, taking the sum of \eqref{eq-Ci-to-be-glued} over $i=1,2$ yields at the level of characteristic sheaves the desired top horizontal map of \eqref{lemma-ghost-glued} up to adding extra summands of $\NN$ to which we map trivially.
In order to form this sum also at the level of actual log structures, we need to lift torsors from $C_i$ to $C$. 
Concretely, let $\bar s_i$ denote the image of $e_i$ under \eqref{eq-Ci-to-be-glued}. We wish to lift $\bar s_i$ to a section $s_i\in(\widehat\shM_C)_\eta$ such that $s_{\one_e}:=s_1\cdot s_2\in (\pi^*\widehat\shM_S)_\eta$.
Once we choose a chart $\widehat Q\oplus_\NN\NN^2\ra (\widehat\shM_C)_\eta$ compatible with the chart \eqref{eq-chosen-chart-X}, the lifts $s_1,s_2$ are given uniquely by the following essential Lemma (cf. \cite[Proposition 7.1]{NS}, \cite{F.Kato},\cite{Li},\cite{Mo}).
\begin{lemma} \label{lemma-unique-extension-torsor}
For a local ring $(A,\fom)$, let $R$ denote the Henselization of $A[x,y]/(xy)$ in the ideal generated by $\fom$ and $x,y$. Let $R_x,R_y$ be the Henselization of $A[x]$,$A[y]$ in $\fom+(x),\fom+(y)$ respectively.
Given $\bar a\in R_x^\times$ and $\bar b\in R_y^\times$, 
there are unique $a,b\in R^\times$ that project to $\bar a,\bar b$ respectively and satisfy the property $ab\in A$.
\end{lemma}
\begin{proof} 
Note that via extension by zero, $x'=\bar ax$ and $y'=\bar by$ define elements in $R$ and we then find the existence of $a,b$ to follow from \cite[Lemma 2.1 ($a=u_x,b=u_y$)]{F.Kato} and their uniqueness is \cite[Lemma 2.2]{F.Kato}.
\end{proof}
Let us now study the dependence on choices. Any other chart \eqref{eq-chosen-chart-X} in reference to the given one has the form $e_1\mapsto a z_1, e_2\mapsto a^{-1} z_2$ for some $a\in B^\times$ which then can be absorbed in an accordingly different chart 
$\widehat Q\oplus_\NN\NN^2\ra (\widehat\shM_C)_\eta$ by multiplying the image of $(0,e_1)$ by $b$ and $(0,e_2)$ by $b^{-1}$ for $b$ a $w_e$'th root of $f^*(a)$. 
This operation leaves $s_{\one_e}$ invariant and it even leaves the inclusion of $(\pi^*\widehat\shM_S)_\eta$ in $(\widehat\shM_C)_\eta$ pointwise invariant.

Next, look at the effect of a change of chart $\widehat Q\oplus_\NN\NN^2\ra (\widehat\shM_C)_\eta$ while keeping the compatibility with \eqref{eq-chosen-chart-X} and also keeping $(\pi^*\widehat\shM_S)_\eta$ invariant (not necessarily pointwise) are given by multiplying the image of $(0,e_i)$ in $(\widehat\shM_C)_\eta$ by some $w_e$'th root of unity $\zeta_i$. As long as $\zeta_1\zeta_2=1$, this leaves $(\pi^*\widehat\shM_S)_\eta$ pointwise invariant but more generally, this acts on 
$(\pi^*\widehat\shM_S)_\eta$ via multiplication by $\zeta_1\zeta_2$. Everything we did is compatible with generizations $\eta\leadsto\eta'$, so we obtain the following result.
\begin{proposition} 
\label{prop-count-choice-of-lift}
Let $i_e:\uS\ra \uC$ denote the section of $\pi:\uC\ra\uS$ that marks the node $e$.
The sheaf of sets on $S$ given by the isomorphism classes of commutative diagrams of log structures (on $\uS$ and $\uC$)
$$
\xymatrix{
 i^{-1}_e\widehat\shM_{C}   & i^{-1}_e(f^*\shM_{X})\ar[l] \\
\widehat\shM_{S} \ar[u]   & \NN\times\shO^\times_S \ar[u]\ar[l] 
}
$$
that lift \eqref{fill-dash-arrow-diagram} along $i_e(S)$ is a torsor under $\mu_{w_e}$ (the ${w_e}$'th roots of unity).
\end{proposition}
Let $\widehat S$ denote the fibre product of the total spaces of the torsors obtained from the edges $e\in E(\tGamma_\rho)$ via Proposition~\ref{prop-count-choice-of-lift}. 
It is a $\prod_{e}\mu_{w_e}$-torsor over $S$, and carries the scheme- and log-structure pulled back from $(\uS,\widehat\shM_{S})$. 
Let $(\widehat C,\widehat\shM_{\widehat C})\ra (\widehat S,\widehat\shM_{\widehat S})$ be the log-smooth curve that is the pullback of $(\uC,\widehat\shM_{ C})\ra (\uS,\widehat\shM_{ S})$ under $\widehat S\ra S$.
\begin{lemma} 
Define $L:=\bigoplus_e \ZZ\one_e $ as a sublattice of $\overline{\widehat \shM}^\gp_{\widehat S}$ and let $L^\sat$ be its saturation.
The inclusion $L \stackrel{\one_e\mapsto s_{\one_e}}{\lra} \widehat\shM^\gp_{\widehat S}$ extends canonically to an injection $L^\sat \ra  \widehat\shM^\gp_{\widehat S}$. 
In particular, also $K^\sat\subset L^\sat$ lifts (see Lemma~\ref{sequence-descend-to-MS}).
\end{lemma}
\begin{proof} 
This is a tautology and follows by construction: 
a point $\hat\eta\in\widehat S$ that lies above $\eta\in S$ is identified with an isomorphism class of charts $\overline{\widehat\shM}_{S,\eta}\ra {\widehat\shM}_{S,\eta}$ whose groupification injects $L$ to ${\widehat\shM}^\gp_{S,\eta}$ as prescribed by the assertion and also maps $L^\sat$ into ${\widehat\shM}^\gp_{S,\eta}$ by mapping the additional elements to roots of products of the $s_{\one_e}$ and the choice of roots is uniquely defined by $\hat\eta$.
\end{proof}
We can now define the quotient $\shM^\gp_{\widehat S}:= \widehat\shM^\gp_{\widehat S}/K^\sat$ 
and obtain $\shM_{\widehat S}=\shM^\gp_{\widehat S}\times_{\oM^\gp_{\widehat S}} \oM_{\widehat S}$ where $\oM_{\widehat S}$ is the pullback of $\oM_{S}$ to $\widehat S$.
We have a surjection $\widehat\shM_{\widehat S}\ra \shM_{\widehat S}$ and for $\shM_{\widehat S}$ to be a log structure, it suffices to show that the structure map $\widehat\shM_{\widehat S}\ra\shO_{\widehat S}$ factors through this surjection.
This follows if we verify that the torsors given by the non-trivial elements in the $\NN^r$-summand in \eqref{hatMS} map to zero in $\shO_S$. 
And indeed, this is because the sections $q_e$ are nowhere zero in $S$, hence all $\shL_{q_e}$ and thus their products and powers map to zero.
We obtain $\shM_{\widehat C}:=\widehat\shM_{\widehat C}\oplus_{\pi^*\widehat\shM_{\widehat S}} \shM_{\widehat S}$ to have a log smooth map $(\widehat C,\shM_{\widehat C})\ra (\widehat S,\shM_{\widehat S})$ that in fact canonically extends as the left column in the diagram \eqref{square}, namely the bottom horizontal map $f_{\widehat S}:\bk\ra\widehat S$ sends the generator of $\NN$ to $s_{\one_e}$ (since with the quotient by $K^\sat$ all $s_{\one_e}$ got identified, we just call their equivalent class $s_\one$, just like all $\one_e$ got identified with $\one$). 
To obtain \eqref{square}, by Proposition~\ref{prop-count-choice-of-lift}, it remains
to argue why and how the top horizontal map in the diagram \eqref{square} is defined away  from the nodes. For simplicity, we omit various decorations of $\widehat{\ }$ on spaces in the following.
I.e., we now sit on only $C_1$ or only $C_2$ and here the map $f$ is completely determined by realizing that, for $\pi_0:X\ra \bk$ denoting structure map of the target,
$$\cM _{X}^{\gp}|_{X_i} = \cM _{X_i}^{\gp}\oplus _{\cO ^{\times}_{X _i}} \pi_0 ^*\cM _\bk^{\gp}|_{X_i}$$ 
and then considering the commutative diagram
\[ \xymatrix{
\cM _C^{\gp}|_{C_i} & \cM _{C_i}^{\gp}\ar@{_{(}->}[l]
 \\
\pi^*\cM _{S}^{\gp}|_{C_i}  \ar[u] &  f ^{*}_i (\cM_{X}^{\gp}|_{X_i}) \ar@{-->}[ul] & f _i^{*} \cM _{X_i}^{gp} \ar@/_/[ul]\ar[l] \\
& (f^{*}_i \pi_0 ^* \cM _\bk^{\gp})|_{C_i}  \ar@/^/^{\pi^*f^*_S}[ul] \ar[u] & \cO ^{\times}_{C_i}\ar[u]\ar[l] } \]
that yields as the dashed arrow away from the gluing nodes a natural homomorphism $(f^\gp)^{*} :\cM^\gp_{X} \ra \cM^\gp_{C}$ compatible with $f_S$ in the sense of \eqref{square}.
Furthermore, the induced map $(\bar f^\gp)^{*}:\oM^\gp_{X} \ra \oM^\gp_{C}$ maps $f ^{*} \oM_{X}$ into $\oM_{C}$ and is the top horizontal map in \eqref{Ghost-square}.
Hence, away from the splitting nodes,  we obtain the desired map $f^*\cM_{X} \ra \cM_{C}$ as the induced map 
$$f^*\shM_{X}=f^*\shM^\gp_{X}\times_{f^*\oM^\gp_{X}}f^*\oM_{X} \ra \shM^\gp_{C}\times_{\oM^\gp_{C}}\oM_{C}=\shM_{C}.$$
We obtained a basic stable log map
\begin{equation}
\vcenter{
\label{desired-gluing}
\xymatrix{
(\widehat C,\shM_{\widehat C})\ar^>>>>f[r]\ar_{\pi}[d]& X\ar^{\pi_0}[d]\\
(\widehat S,\shM_{\widehat S})\ar_>>>>{f_{\hat S}}[r]& \bk
}}
\end{equation}
that lifts \eqref{Ghost-square} and when applying the splitting construction of the previous section (up to taking the quotient $\hat S\ra S$) gives back the curves $f_i:C_i/S\ra X/\bk$ that we started with.

To conclude this section, it remains to observe that when producing $\hat S$, we marked a bit too much. Indeed \eqref{desired-gluing} has a non-trivial group of automorphisms, namely $\mu_l$ acting by pullback of the left column of the diagram along endomorphisms of the log structure
$\shM_{\hat S}\ra \shM_{\hat S}$ that modify a chart by pointwise fixing $\shF_{\widehat S}$ and multiplying the image of $q$ by $\zeta \in \mu_{l}$ where $q$ is the element found in the proof of Lemma~\ref{images-generate-MS}, i.e., so that $\oM^\gp_{\widehat S}=\oF^\gp_{\widehat S}\oplus q\ZZ$.
If $\rho\in \oM_{\widehat S}^\vee[1]$ is the primitive generator whose perp is $\oF_{\widehat S}$ then $\rho(\one)=l$ by \S\ref{section-graphs-to-curves} and $\rho(q)=1$.
This shows that the described action fixes $s_\one$. Furthermore, it acts transitively on the sheets of $\widehat S\ra S$ by means of the injection $\mu_l\ra \prod_e \mu_{w_e}, \zeta\mapsto \zeta^{l_e}$ (this injects because the $l_e$ are coprime since $l=\lcm(\{w_e\})$ and $l_e=l/w_e$).
We denote the quotient by $\tilde S:=\widehat S/\mu_l$ and let $\tilde C$ be the quotient log smooth curve above it. We observe
\begin{equation}
\label{gluing-degree}
\deg(\tilde S\ra S)=\frac{\prod_e w_e}{\lcm(\{w_e\})}.
\end{equation}


\section{The splitting stack}
\label{sec-splitting-stack}
The purpose of this section is to first recall Olsson's stack $\widetilde{\cL og} _{S}$ which classifies maps of fine log schemes $T\ra S$ for fixed $S$. 
We then introduce a new stack $\Spl$ that surjects to $\widetilde{\cL og} _{\bk}$ where $\bk$ denotes the standard log point. 
A typical local model for $\widetilde{\cL og} _{\bk}$ is given by an injection of monoids $h:\NN\ra Q$. Subject to $h$, a typical local model for $\Spl$ is given by a choice of facet in $Q$ with the property that its intersection with the image of $h$ is trivial.

For a fine log scheme $S$, denote by $\widetilde{\cL og} _{S}$ the Artin stack over $\uS$ due to Olsson \cite{Ol} that is defined as follows. 
The objects over a scheme morphism $\underline{T}\ra \uS$ are
the morphisms $T\ra S$ of fine log schemes over $\underline{T}\ra \uS$. The  
morphisms from $T\ra S$ to  $T'\ra S$ are the log morphisms $h: T\ra T'$ over $S$ for which $h^*\cM ' \ra \cM$ is an isomorphism.  
The stack $\widetilde{\cL og }_{S} $ is
an algebraic stack locally of finite presentation over $\uS$ (see  \cite[Theorem 1.1]{Ol}). 
Let $\cL og _{S}$ be the open substack of $\widetilde{\cL og }_{S}$ 
classifying fs log schemes over $S$ (see  \cite[Remark 5.26]{Ol}). 

\begin{definition}\label{splStack} 
Recall that $\bk=\Spec (\NN\stackrel{\tiny 1\mapsto 0}{\lra} \kk)$ denotes the standard log point.
We denote by $\Spl$ the category fibred in groupoids over the category $( \mathrm{Sch}/ \ubk ) $ of schemes over $\ubk=\Spec \Bbbk$
whose fibre over $\uT\ra \ubk$ 
is the groupoid of triples $$(T, h, \cF)$$ where $(T, h:\NN_T\oplus \Bbbk_T^{\times} \ra\cM)$ is an object in $\cL og _{\bk}$ and
$\cF $ is a subsheaf of $\cM $ satisfying that:
\begin{enumerate}

\item  For every $t\in \underline{T}$, $\cF _{\bar{t}}$  is a facet of  $\cM _{\bar{t}}$ (i.e., $ab\in \cF _{\bar{t}}\Rightarrow a,b\in \cF _{\bar{t}}$).

\item\label{logF} For the log structure $\ka : \cM  \ra \cO _T$,  $\alpha _{|_{\cM  \setminus \cF }}= 0$.

\item\label{p_extremal} 
For every $t\in \underline{T}$, $\langle \oF_{\bar{t}},\one\rangle^\gp\ot_\ZZ\RR = \oM_{\bar{t}}^\gp\ot_{\ZZ}\RR$ 
where $\one$ is the image of $1$ under the induced homomorphism $\overline{h}_{\bar{t}}:\NN_{T,\bar{t}}\ra\oM_{\bar{t}}$ 
and $\langle\oF_{\bar{t}},\one\rangle$ is the submonoid of $\oM_{\bar{t}}$ generated by $\oF_{\bar{t}}$ and $\one$.

\end{enumerate}
The morphisms from $(T, h, \cF)$ to $(T', h', \cF ')$ are the morphisms from $(T, h)\ra (T', h')$ in  $\cL og _{\bk}$
 for which $\cF$ goes to $\cF'$.
\end{definition}

Note that by Condition \eqref{logF} the pair $(\cF , \alpha _{|_{\cF }})$ is also a log structure on $\uT$.
Note also that instead of $\cF$ we may give a sheaf of facets $\oF\subset \oM$ with suitable properties because by (1) we have that $\cF$ is the inverse image of its projection $\oF$ in $\oM$.

\medskip

It is straightforward to check that $\Spl$ is a stack over $\ubk$.
Below, we will show that the forgetful morphism 
\begin{eqnarray*} \Spl \ra \cL og_{\bk} , \  \ (T, h, \cF) \mapsto (T, h)  \end{eqnarray*}
is a representable, proper, and \enquote{normalization} map (see Proposition \ref{Prop_Normal}).

Let $Q$ be a toric monoid and $q$ be a nonzero element of $Q$. Consider $\Spec (Q\ra \kk[Q]/(q))$, where $(q)$ denotes 
the ideal of $\kk[Q]$ generated  by the character $\chi^q$.  
\begin{lemma} \label{lem-smooth-atlas-piece}
By taking $T=\Spec (Q\ra \kk[Q]/(q))$ and $\NN\hra Q,\ 1\mapsto q$ for $\overline{h}$ there is a morphism 
$$\Spec  (\kk[Q]/(q) ) \ra \cL og _{\bk}$$
that is representable and smooth in the ordinary sense.
\end{lemma}
\begin{proof}
By the toroidal characterization of log smoothness \cite[Theorem 3.5]{K.Kato}, the map $T\ra\Spec\,\bk$ is log smooth.
By the classifying properties of $\cL og _{\bk}$, \cite[Theorem 4.6 (ii) and Cor 5.31]{Ol}, we find that the map in the assertion is smooth.
The map is representable by \cite[\S4, Rem. 4.2]{Ol}.
\end{proof}

Recall the convention $Q^\vee=\Hom_{\hbox{\tiny Mon}}(Q,\NN)$ and that 
$\rho^{\perp } = \{ p \in Q \ | \ \rho (p) = 0 \}$ for $\rho \in Q^\vee [1]$ gives a bijection between facets of $Q$ and $Q^\vee [1]$.
                               
\begin{lemma} 

\begin{enumerate}

\item 
The reduced scheme $\Spec  (\kk[Q]/(q) ) ^{\mathrm{red}}$ of\/ $\Spec  (\kk[Q]/(q) )$ is canonically isomorphic to 
$$ \bigcup _{\rho \in Q^\vee[1]: \rho (q)\ne 0}\Spec  (\kk[Q]/(Q\setminus \rho^{\perp} )) , $$
the union of the closed subschemes of \/ $\Spec  (\kk[Q]/(q) )$ defined by the ideals generated by $Q\setminus \rho^{\perp}$ for varying $\rho$. 

\item 
The fibre product $\Spec  (\kk[Q]/(q))\times_{\cL og _{\bk}} \Spl$ is representable 
by the disjoint union $$ \coprod _{\rho \in Q^\vee[1]: \rho (q)\ne 0} \Spec  (\kk[Q]/(Q\setminus \rho^{\perp}  )) $$ of irreducible components of $\Spec  (\kk[Q]/(q) ) ^{\mathrm{red}}$.

\end{enumerate}
\label{Normal} 

\end{lemma}

\begin{proof} 
Let $\uT:=\Spec  (\kk[Q]/(q) )$ and $\uT _\rho :=  \Spec  (\kk[Q]/(Q\setminus \rho^{\perp} ))$.
Recall the well-known fact that
the divisor $\mathrm{div}(\chi ^q)$ of the character $\chi ^q$ as a rational function on $\Spec (\kk [Q])$ 
is $\sum _{\rho \in Q^{\vee}[1]} \rho (q) \uT _\rho $.
This proves (1).

For (2), we first construct a natural $\uT$-morphism 
\begin{equation} \label{eq-Trho-chart}
\coprod _{\rho : \rho (q) \ne 0} \uT_\rho \ra \uT \times_{\cL og _{\bk}} \Spl.
\end{equation}
Note that the affine coordinate ring of $\uT _\rho$ has
two expressions $\kk[\rho ^{\perp} ]$ and $\kk[Q]/(Q\setminus \rho ^{\perp} )$ which are isomorphic via the inclusion $\rho^{\perp} \subset Q$.
Hence on $\uT$ we can consider two induced log structures:
$\cF$ defined by $\Spec (\rho^{\perp} \ra \kk[\rho^{\perp}])$ and $\cM _{T_\rho}$ defined by $\Spec (Q \ra \kk[Q]/(Q\setminus \rho^{\perp}))$.
For $t \in \uT _\rho$, denote by  $\alpha _{t}$ the natural homomorphism $Q\ra \cO _{\uT _\rho, \bar{t}}$. 
Then $\alpha _t^{-1}(\cO ^\times _{\uT_\rho, \bar{t}}) \subset \rho^{\perp}$, so since $\rho^\perp\subset Q$ is a facet, 
$$ 
\overline{\cF _{\bar{t}} }\ \cong\ \rho^{\perp} /\alpha _t^{-1}(\cO ^\times _{\uT_\rho, \bar{t}})
\  \subset\  
Q/\alpha _t^{-1}(\cO ^\times _{\uT_\rho, \bar{t}})\ \cong\ \overline{\cM}_{T_{\rho}, \bar{t}}
$$ is also one.
Let $\iota _\rho$ denote the inclusion $\uT _\rho \ra \uT $. 
The pair $(\iota _{\rho}, \cF \ra  \cM _{T_\rho} \leftarrow  \NN_{\uT _{\rho}} ) $ can be considered 
as an object of $(\underline{T}\times_{\cL og _{\bk}} \Spl ) (\uT_\rho)$.

Conversely, we construct a natural $\uT$-morphism from $\underline{T}\times_{\cL og _{\bk}} \Spl$ to $\coprod_{\rho : \rho (q) \ne 0} T_\rho$.
Suppose that we are given a morphism $\uS \ra \underline{T}\times_{\cL og _{\bk}} \Spl$ by data
$h:\uS\ra\uT$, $\NN_S\oplus\Bbbk^{\times}_S\ra \cM_S\leftarrow\cF_S$ 
such that $\NN_S\oplus\Bbbk^{\times}_S\ra\cM_S$  is the pullback of  $\NN_T\oplus\Bbbk^{\times}_T\ra\cM_T$ under $h$. 
Suppose $\uS$ is connected.
Fix $s\in S$, the composition $g_s:Q\ra \cM_{\uT,h(\bar s)}\ra \cM_{S, \bar{s}}$ induces a chart, so $Q_0:=g_s^{-1}(\cM_{S, \bar{s}}^\times)$ is a face of $Q$ and $g_s$ induces
an isomorphism $Q/Q_0\stackrel\sim\lra \oM_{S, \bar{s}}$. 
In particular, $g^{-1}_s(\oF_{S, \bar{s}})$ is a facet $\rho^\perp$ of $Q$ for a unique $\rho$ with $\rho (q) \ne 0$.
Since $g_s$ is compatible with specialization, $\rho^{\perp}$ is independent of the choices of $s\in S$.
By \eqref{logF} of Definition~\ref{splStack}, $h:\uS\ra\uT$ factors through $\uS\ra \uT_\rho\ra\uT$.  

The above two natural morphisms are inverse to each other. 
 \end{proof}

\begin{proposition}
\begin{enumerate} 
\item The morphism $$\coprod _{{Q,q,\rho}} \uT_\rho\ra \Spl$$ obtained via \eqref{eq-Trho-chart} is representable, smooth, and surjective. Here, the disjoint union runs over all toric monoids $Q$ with nonzero element $q$ and
$\rho \in Q^\vee[1]$ such that $\rho (q)\ne 0$.

\item  The fibred category  $\Spl$ is a pure zero-dimensional algebraic stack over $\ubk$.

\item The forgetful morphism $\Spl \ra \cL og_{\bk}$ is representable, affine, proper and surjective.
Every map $V\ra \cL og_{\bk}$ from a normal variety $V$ factors uniquely through the forgetful morphism $\Spl \ra \cL og_{\bk}$.
\item  Given an object $(T,h:\NN_T\oplus\kk^\times_T\ra\shM_T)$ of $\cL og_{\bk}$, the inverse image in $\Spl$ is in natural bijection with the set 
$$\left\{\rho\in\Gamma(T,\oM_T^\vee)\,\mid\,\rho(\bar h(1))\neq 0, \rho\in\oM_{T,\bar t}^\vee[1]\hbox{ for all }\bar t\in T\right\}.$$
\end{enumerate}
\label{Prop_Normal}
\end{proposition}

\begin{proof}
By \cite{Ol}, Theorem 1.1 and Remark 5.26, see also page $777$ in loc.cit.~for the notations $\mathcal{S}_P$, $\mathcal{S}_Q$ in Corollary 5.25, we conclude that $\cL og_{\bk}$ is an algebraic stack and
\begin{equation} \label{eq-Olsson-etale-cover}
\coprod _{Q,q}
[\Spec ( \kk [Q] ) /\Spec (\kk[Q^{gp}]) ]\times_{[\Spec ( \kk [\NN] ) /\Spec (\kk[\ZZ]) ]} \Spec(\kk) \ra \cL og_{\bk}
\end{equation}
is representable, \'etale, and surjective
where $\NN\ra Q$ is given by $1\mapsto q$.
The map in Lemma~\ref{lem-smooth-atlas-piece} factors through this etal\'e cover, hence
$$\coprod _{Q,q} \Spec  (\kk[Q]/(q) ) \ra \cL og_{\bk}$$ is smooth, representable and surjective.
Assertion (1) now follows by base change to $\Spl$ via Lemma~\ref{Normal},\,(2).

The claim in (2) that $\Spl$ is an algebraic stack now follows from \cite[Lemma C.5]{AOV}. 
The stack $\cL og_{\bk}$ is pure zero-dimensional because the left hand side in \eqref{eq-Olsson-etale-cover} is pure zero-dimensional.
By part (2) of Lemma~\ref{Normal}, the map $\Spl\ra \cL og _{\bk} $ is finite and surjective, so $\Spl$ is also pure zero-dimensional.

Part (3) is similarly a direct consequence of Lemma~\ref{Normal}. 

To show (4), first consider the special case where $T$ has a global chart, say $T=\Spec(Q\ra R)$ for some $\kk$-algebra $R$ and the map to $\bk$ given by $q\in Q$, then the natural map $T\ra \cL og_{\bk}$ factors canonically through the map given in Lemma~\ref{lem-smooth-atlas-piece} and the statement of (4) follows directly from Lemma~\ref{Normal}~(2).
The general case reduces to the special case by choosing an atlas and checking that the statement of (4) is compatible with localization.
The latter follows from the discussion in \S\ref{sec-specialization}: the collection of elements $\rho$, one for each chart, naturally glues to global section of $\oM_T^\vee$.
\end{proof}


\section{Decomposing moduli stacks of curves} \label{section-decomp-curve-moduli}
Let $\mathscr{M}_{g, n}$ 
denote the moduli stack of prestable curves with its natural log smooth structure \cite{F.Kato},\cite[Appendix A]{GS} over the trivial log point $\ubk$. 
Using the splitting stack introduced in the previous section, we define
\begin{gather*}
 \fM  := \cL og _{\mathscr{M} _{g, n}}\times _{\cL og _{\ubk}} \cL og _{\bk}  , \quad    \fM ^{\spl}:= \PY,  \\
 \mathscr{M}  :=\mathscr{M}_{g, n}(X/\bk, \beta).
\end{gather*}  
Objects of $\mathscr{M}$ are diagrams of the form \eqref{square} but with $W=\bk$. Forgetting the right column in the diagram defines a map of log stacks
$\mathscr{M}\ra\mathscr{M}_{g, n}$ and forgetting the top row in the diagram defines a map to $\bk$. 
Taken together, we obtain a log map $\mathscr{M}\ra \mathscr{M}_{g, n}\times_\ubk \bk$ and thus a map
$\underline{\mathscr{M}}\ra \fM$ that we call \emph{forget-target-morphism}. 
By abuse of notation, we will omit the underline on $\mathscr{M}$ and write $\mathscr{M}\ra \fM$ for this map.
Since $\mathscr{M}_{g, n}$ is of pure dimension $3g-3+n$, the same holds true for $\cL og _{\mathscr{M} _{g, n}}$ by \cite[Corollary 5.25]{Ol}. Proposition~\ref{Prop_Normal} now implies that $\fM$ and $\fM ^{\spl}$ are of pure dimension $3g-3+n$ as well.

If $C\ra S$ is an object in $\fM ^{\spl}$ then, by Definition~\ref{splStack}, we have a global section $s_\one:=h(1)\in\Gamma(S,\shM_S)$ whose image in $\Gamma(S,\oM_S)$ we denote by $\one$.
We also have, by Proposition~\ref{Prop_Normal}\,(4), $\rho\in\Gamma(S,\oM_S^\vee)$ with $l:=\rho(\one)\neq 0$.
Furthermore, we have a map $\shM^{C/S}_S\ra \shM_S$ from the natural log structure $\shM^{C/S}_S$ on $ \mathscr{M}_{g, n}$ to the one given with the object. 
If $e$ is a node of a fibre of $C\ra S$ over a point $s\in S$, there is an $\NN$-summand in $\oM^{C/S}_{S,s}$ and its generator maps to an element $q_e$ in $\oM_{S,s}$. 
We call the node $e$ a \emph{splitting node} if $l_e:=\rho(q_e)\neq 0$ and note that if this is the case, this node doesn't ever get smoothed anywhere over $S$ (since $q_e$ then generates $(\oM_S/\rho^\perp)\otimes\QQ$ which is nowhere trivial on $S$ since $\rho(q_e)\neq 0$). If $\bar\Gamma_\rho$ denotes the graph obtained from collapsing all non-splitting edges in a fibre of $C\ra S$, we find that all fibres of $C\ra S$ are marked by this graph.
Adding markings $n_V$ and genera $g_V$ as obtained from $C$ as decorations to the graph $\bar\Gamma_\rho$, we obtain part of the data of \S\ref{sec-graphs}. 
We set $w_e:=\frac{l}{l_e}$ which may be rational a priori.
For fixed $g,n$, let $\Omega(g,n)$ denote the set of graphs with vertices decorated by $g_V$, $n_V$ and edges with rational $w_e$ satisfying \eqref{genus} and \eqref{partition_n}.

The stack $\fM$ is too big for our purpose of arriving at the statements of Lemma~\ref{lem-characterize-components} and Lemma~\ref{lemma-degree-of-mu} below.
To be able to state these lemmata, we will therefore choose a sufficiently small open substack $\fM_0$ with the property that the forget-target-morphism $\mathscr{M}\ra \fM$ factors through the open embedding 
$\fM_0\subseteq \fM$ as follows. Let $P:M\ra \fM$ be a smooth presentation. For every point $x$ in the image of $\mathscr{M}\ra \fM$, let $\hat{x}\in M$ be a point so that $P(\hat x)=x$ and let $U_x\ra M$ be an \'etale neighbourhood of $\hat x$ that is a chart for the log structures at $\hat x$. We define $\fM_0$ to be the union of $P(U_x)$ for $x$ running through the points in the image of 
$\mathscr{M}\ra \fM$. We define $\fM_0^{\spl}:=\fM_0\times_{\fM} \fM ^{\spl}$.

For fixed $g,n$, let $\Omega(g,n)$ denote the set of graphs with vertices decorated by $g_V$, $n_V$ and edges with $w_e$ satisfying \eqref{genus} and \eqref{partition_n} and let
$$\bar \Omega(g,n,\beta):=\im(\Omega(g,n,\beta)\ra \Omega(g,n))$$
be the (finite) image of the map that forgets $\beta_V$. 
For $\Gamma \in \bar \Omega(g,n,\beta)$
let $\fMG$ be the open and closed substack of $\fM_0^{\spl}$ whose points are marked by $\Gamma$ by the discussion in the preceding paragraphs.

\begin{lemma}
\label{lem-characterize-components}
$\fM_0^{\spl} = \coprod_{\Gamma\in\bar \Omega(g,n,\beta)} \fMG$ 
\end{lemma}
\begin{proof} 
We only need to show that $\fM_0^{\spl}$ is contained in the right hand side. 
Every point $x$ in $\fM_0$ lies in a chart $U_y$ of the log structure of some point $y$ that is contained in the image of 
 $\mathscr{M}\ra\fM_0$. 
 This implies that $\oM_{U_y,x}^\vee$ is a face of $\oM_{U_y,y}^\vee$ and for the latter we have an interpretation as a parameter space of integral tropical curves by Lemma~\ref{trop-moduli-spaces}. 
 The points of the inverse of $y$ in $\fM_0^{\spl}$ are in bijection with the elements of $\oM_{U_y,y}^\vee[1]$ that evaluate non-trivial on $\one$ by Proposition~\ref{Prop_Normal}\,(4). 
  By \eqref{graph-to-curve}, these same elements yield elements of $\Omega(g,n,\beta)$ under $\Trop$ and so the corresponding graphs with $\beta_V$ forgotten lie in $\bar \Omega(g,n,\beta)$.
By \S\ref{sec-specialization}, the points in the inverse image of $x$ in $\fM_0^{\spl}$ are indexed by the subset $\oM_{U_y,x}^\vee[1]$ of $\oM_{U_y,y}^\vee[1]$ which are therefore also labelled by graphs in $\bar \Omega(g,n,\beta)$. 
\end{proof}

We use the composition
$$\mu _{\fMG}\colon \fMG\ra \fM_0^{\spl} \ra \fM_0$$
to define
\begin{equation} \label{def-KGamma}
\mathscr{M}_{\Gamma}    := \mathscr{M} \times _{\fM_0}  \fMG.
\end{equation}
Let $\mu_\Gamma\colon \mathscr{M}_{\Gamma}\ra \mathscr{M}$ denote the forgetful map. 
By Proposition~\ref{Prop_Normal},\,(4) and Lemma~\ref{lem-characterize-components}, the map $\coprod_\Gamma\mathscr{M}_{\Gamma}\stackrel{\coprod\mu_\Gamma}\ra \mathscr{M}$ is surjective.
We arrive at a commutative diagram with Cartesian squares
\begin{equation} \label{Cartesian-curve-moduli-split}
\vcenter{
\xymatrix{  \coprod _{\Gamma \in \bar\Omega (g, n, \beta ) } \mathscr{M} _{\Gamma} \ar[r]  \ar[d]^{\coprod\mu _{\Gamma} } 
&     \coprod _{\Gamma \in  \bar\Omega (g, n, \beta ) } \fMG  \ar[d]^{\coprod \mu  _{\fMG }} \ar[r] &   \fM ^{\spl} \ar[d] \ar[r] 
     &     \cL og ^{\mathrm{spl}}_{\bk}  \ar[d]  \\
 \mathscr{M} \ar[r] &  \fM _0  \ar[r]_{\mathrm{open}\ } 
 &  \fM \ar[r]_{ \mathrm{pr}}  &   \cL og _{\bk}. } }
\end{equation}
Let $\pi: \shC \ra \mathscr{M} $ be the universal curve for $\mathscr{M}$ and $f:\shC\ra X$ be the universal map.  
Recall that  
$$
(R\pi _{*}( f^* \shT_{X/\bk}))^\vee\ra \LL_{\underline{\mathscr{M}}/\underline\fM_0} 
$$
is the natural perfect obstruction theory for $\mathscr{M}$ relative to $\fM_0$ (see \cite[\S5]{GS}).
Then similarly $\mu _{\Gamma}^* (R \pi _{*} (f^* \shT_{X/\bk})^\vee$ defines a perfect obstruction theory for $\mathscr{M} _{\Gamma}$
relative to $\fM _{\Gamma}$, since the intrinsic normal cone of $\mathscr{M}_{\Gamma}$ relative to $\fM_{\Gamma}$ is closely immersed to the pullback of 
the intrinsic normal cone of $\mathscr{M}$ relative to $\fM _0$ by the Cartesianness of \eqref{Cartesian-curve-moduli-split}, Proposition~\ref{Prop_Normal} and Lemma~\ref{Normal}.

\begin{lemma} 
\label{lemma-degree-of-mu}
 Under the projective morphisms $\mu_{\fM_{\Gamma}}$, as an identity in $A_{3g-3+n}(\fM _0)$, we have
 \[ \sum _{\Gamma \in  \bar\Omega (g, n, \beta)}  l_\Gamma (\mu_{\fM _{\Gamma}})_*[\fM _{\Gamma}] = [\fM _0].  \]  
\end{lemma}

\begin{proof} 
The morphisms $\mu_{\fM_{\Gamma}}$ are affine and proper by Proposition~\ref{Prop_Normal},\,(3) and thus projective.
By the Cartesianness of \eqref{Cartesian-curve-moduli-split}, it is enough to prove the corresponding statement for the forgetful morphism $\Spl \ra\cL og_\bk$. 
The statement can be checked on the map of presentations as given in Proposition~\ref{Prop_Normal}. 
The components of the presentation are studied in Lemma~\ref{Normal} and so the claim follows if we show that the cycle $[\Spec  (\kk[Q]/(Q\setminus \rho^{\perp} ))]$ appears with coefficient $l_\Gamma$ in $[\Spec\kk[Q]/(\one)]$. 
By assumption, $l_\Gamma=\rho(\one)$ and by standard toric geometry $l_\Gamma$ is the vanishing order of the character $\chi^\one$ on $\Spec  (\kk[Q]/(Q\setminus \rho^{\perp} ))$, so we conclude the proof.
\end{proof}

Consider $\fM _{\Gamma}' := \fM_{\Gamma} \times _{\ubk} \Spec (\kk[x]/(x^{l_\Gamma}) )$ in order to have the induced projective morphism $\mu':\coprod _{\Gamma} \fM _{\Gamma}' \ra \fM _0$ be of pure degree one.
Hence, by applying \cite[Proposition 5.29]{Man} (see also \cite{HW}) to $\mu'$, we get 
\begin{lemma}\label{SplFund}  
$
\begin{array}{l} \sum _{\Gamma  \in \bar{\Omega}(g, n,\beta ) }  l_{\Gamma} (\mu _{\Gamma})_*
 [ \mathscr{M} _{\Gamma} / \fM_{\Gamma},  \mu _{\Gamma}^* (R \pi _{ *} f^* \shT_{X/\bk})^\vee ]  = [\mathscr{M} / \fM _0 ,  (R \pi _{ *} f^* \shT_{X/\bk})^\vee ]. 
\end{array} 
$
\end{lemma}
Here and later, $[K/M, E ]$ (or sometimes simply $[K, E]$)
denotes the virtual fundamental class of a stack $K$ that is relative DM type over a pure dimensional algebraic stack $M$,
with respect to a relative perfect obstruction complex $E$ of $K$ over $M$ (see \cite{LG, BF, Kr}).


\section{Comparing perfect obstruction theories} \label{section-ob-theories}
In this section, we deliver the details for \S\ref{subsec-cycle-degen-formula}. 
Recall the forgetful maps of finite graph sets 
$$ \tilde\Omega(g, n,\beta) \ra \Omega(g, n,\beta) \ra \bar\Omega(g, n,\beta)  $$
where the first map forgets the edge ordering and the second map forgets the $\beta_V$.
\subsection{Gluing of log structures and moduli stacks} 
For $\bGamma\in \bar \Omega(g,n,\beta)$, recall $\mathscr{M}_{\bGamma}$ and $\foM_{\bGamma}$ from \eqref{def-KGamma}. 
We decompose $\mathscr{M}_{\bGamma}$ in the obvious way into open and closed substacks $\mathscr{M}_{\Gamma}$ that fix the classes $\beta_V$, i.e.,
\[ \mathscr{M}_{\bGamma} = \coprod _{\Omega(g,n,\beta)\ni\Gamma \mapsto \bGamma} \mathscr{M} _{\Gamma}.\]
Furthermore, we may consider the edge-labelled stack $\mathscr{M}_{\tGamma}$ where splitting nodes are marked. 
We have an \'etale forgetful map of the markings $\mathscr{M}_{\tGamma}\ra \mathscr{M}_{\Gamma}$. 
We similarly define $\fM_{\tilde{\Gamma}}$ with an \'etale forgetful map $\fM_{\tilde{\Gamma}}\ra \fM_{\Gamma}$.
  Let $\mu _{\tGamma}$ denote the composition 
 \[   \mathscr{M}_{\tGamma}\ra \mathscr{M}_{\bGamma} \xrightarrow{\coprod \mu _{\Gamma}}  \mathscr{M} . \]
As a direct consequence of Lemma~\ref{SplFund}, we arrive at the following. 
\begin{lemma} \label{SplFund2}  
$$\begin{array}{l} 
\sum _{\tGamma  \in \tilde{\Omega}(g, n,\beta) }  \frac{l_{\Gamma}}{|E(\tGamma)|!} (\mu _{\tGamma})_*
 [ \mathscr{M} _{\tGamma} / \fM_{\tGamma},  \mu _{\tGamma}^* (R \pi _{ *} f^* \shT_{X/\bk})^\vee ]  = [\mathscr{M} / \fM _0 ,  (R \pi _{ *} f^* \shT_{X/\bk})^\vee ]. 
\end{array} $$
\end{lemma}
For the remainder of the section, we fix $\Gamma\in \tilde{\Omega}(g, n,\beta)$, in particular, edges of the graph $\Gamma$ will be ordered from now on. 
Recall that for a vertex $V$ we defined $\mathscr{M}_V:= \mathscr{M}_{g_V,n_V\cup E_V} (X_{r(V)},\beta_V)$, i.e., the moduli stack of genus $g_V$ basic stable log maps to $X_{r(V)}$ of class $\beta_V$ with $n_V$ unconstrained markings and further markings indexed by $e\in E_V$ with contact order $w_e$ to $D$. 
We also consider the stack $\mathscr{M}^\circ_V:=  \mathscr{M}_{g_V, n_V \cup E_V}$ of prestable curves, see the beginning of \S\ref{section-decomp-curve-moduli}.

The diagram 
\begin{equation}
\vcenter{ \label{diag-gluing-moduli}
 \xymatrix{ \mathscr{M}_{\Gamma} \ar[d]_{\phi_{\Gamma}} \ar[r]^{w} & \fM _{\Gamma} \ar[rd]^{\mff} \\
    \bigodot_V \mathscr{M}_V \ar[r]_{w'\quad}   &   \prod _V    \cL  og_{  \mathscr{M}^\circ _V} \ar[r]_{\mfs\quad}       &     
      \cL og_{ \Pi _V  \mathscr{M}^\circ _V}  =:\fB
}}
\end{equation}
commutes, where $\phi_{\Gamma}$ is defined by the splitting construction of \S \ref{splitting}; $w$, $w'$ are morphisms that forget the targets of stable log maps; $\mff$ is defined by taking the homomorphism $\oplus _V \cM ^{\uC_V/\uS}_{\uS} \ra \cF ^{\uC/\uS}_{\uS}$ (i.e., taking only the ``facet data'') and finally $\mfs$ is defined by taking the sum $\oplus _V \cM ^{\uC_V/\uS}_{\uS} \ra \oplus _V \cM _{S_V}$.

Recall that $D$ carries the trivial log structure.
We define the $\cO _D$-module  $N_{D/X_{i}}$ by the exact sequence 
\begin{equation}  \label{log-normal-to-D}
0 \ra N_{D/X_{i}} \ra \shT_{X_i} |_{D} \ra \shT_{D} \ra 0 . 
\end{equation}
Note that $N_{D/X_{i}} $ is isomorphic to $\shO_D$ (unlike $N_{D/\uX_{i}}$).

\begin{lemma}\label{etale}
\begin{enumerate}

\item\label{wprime} $\mfs$ is \'etale.

\item\label{smooth} The morphism $\mff$ is of Deligne-Mumford-type and smooth.

\item\label{dis}  \[ w^* \LL ^{\vee}_{\fM_{\Gamma} /  \fB} \cong \phi _{\Gamma}^* \Bigg(\bigoplus _{e\in E(\Gamma)} \ev^*_e N_{D/X_{1}}\Bigg)  \] 
where $\fB=\cL og_{\prod _V {\mathscr{M}^\circ _V}}$.

\item\label{phi} The morphism $\phi_{\Gamma}$ is of Deligne-Mumford-type and \'etale of degree $\frac{\Pi _{e\in E(\Gamma)} w_e }{ l_{\Gamma} }$.
 Here, if $\Gamma$ has only one vertex, then we set $\Pi _{e\in E(\Gamma)} w_e  = 1$.

\end{enumerate}
\end{lemma}

\begin{proof}
For \eqref{wprime}: This is  straightforward by  the lifting criterion for formally \'etale morphisms.

For \eqref{smooth}: First note that $\mff$ is of DM-type since there is no infinitesimal automorphism $\sigma$ of a geometric point of $\fM _{\Gamma}$
with $\mff (\sigma ) = \mathrm{id}$. Now to prove $\mff$ is smooth, 
 it is enough to show $\mff$ is formally smooth since it is locally of finite presentation. 
 The corresponding lifting property of $\mff$ can be checked by considering
charts of log morphisms. Let $I$ be a nilpotent ideal of  a finitely generated ring $\Lambda$ over $\mathbf{k}$
and let $\uS=\Spec (\Lambda/I)$. We may assume that there is a chart 
\[ \xymatrix{ \overline{M}_{S}^{C/S}  \ar[rd] \ar[r] &  Q\ar[d] & \ar_{\one\mapsfrom 1}[l] \ar[ld] \NN \\
            &   \mathcal{O}_S & } \]    of $\cM ^{C/S}_{S} \ra \cM _S \leftarrow \NN \oplus \mathcal{O}_S^\times $.
 By Definition \ref{splStack}, there exists a unique lifting to $\Spec (\Lambda)$ of the log structure on $\cM _S$. It is also
 obvious that $\cM _{S}^{C/S}\ra \cM _S$ and $\cM _S \leftarrow \NN \oplus \mathcal{O}_S^\times$ have lifts, which may not be unique. 

For \eqref{dis}: Since $N_{D/X_1}$ is a trivial line bundle, 
it is enough to show that $\LL ^{\vee}_{\fM _{\Gamma} / \fB}$ is also a  trivial bundle
of rank $|E(\Gamma)|$. For this we will describe an $\cO _S$-basis of the set of {\em isomorphism} classes  
of liftings  to $S[\epsilon]:=\Spec (\cO_S[\epsilon]/\epsilon^2)$, fixing the log structure 
of the facet, of a given object $$(C/S, \NN\oplus\cO^{\times}_S\stackrel{h}{\longrightarrow}\cM_S\stackrel{\ j}{\longleftarrow}\cM^{C/S}_S)$$  of $\fM_{\Gamma}$ over any
scheme $\uS$.
For each $e\in E(\Gamma)$, we define a lifting
\begin{align}\label{lifting} (C[\epsilon]/S[\epsilon],
\NN \oplus(\cO_S[\epsilon])^{\times}\stackrel{h+\epsilon h'}{\longrightarrow}\cM_S[\epsilon]\stackrel{\ j_e}{\longleftarrow}\cM^{C/S}_S[\epsilon]) \end{align} as follows.
The monoid homomorphism of $h': \NN \ra \cO _S$ can be transformed to the trivial homomorphism up to a unique isomorphism of $\cM _S$ fixing $\cF _S$,
since $\cO _S$ is a divisible group, and $\oF$ and $\NN$ generate $\oM\ot_{\ZZ}\QQ$. 
For each splitting node $q$, there is the corresponding canonical submonoid sheaf 
$\shN_q\subset\cM^{C/S}_S$
satisfying $\shN_q \cong\NN_q\oplus\cO^{\times}_S$ by Lemma~\ref{lemma-section-exists} which deals with $\shL_{\one_q}\cong 1\oplus \cO^{\times}_S$.
Therefore $j_e$ is determined by $j_e(1_q)$ for all splitting nodes $q$. Define   $j_e(1_q)=j(1_q)$ for $q\ne e$, $j_e(1_q)=j(1_q)+ \epsilon$ for $q=e$.
We can check that $j_e$ is well-defined and the liftings \eqref{lifting} for $e\in E(\Gamma)$ form a basis of the set of the isomorphism classes of liftings.
We can do this construction for a smooth surjective cover $\uS \ra \fM_{\Gamma}$.

Finally, \eqref{phi} follows from the gluing construction in \S\ref{gluing} that gave the preimage $\tilde S$ of any scheme $S$ mapped to the target of $\phi_\Gamma$ and \eqref{gluing-degree} readily computes the degree of $\phi_\Gamma$.
\end{proof}

Let us fix $\Gamma\in\tilde\Omega(g, n,\beta)$ and for a vertex $V$ we define
\begin{equation} \label{def-MV}
\fM_V:=\cL og _{ \mathscr{M} _{g_V, n_V\cup E_V}}.
\end{equation} 
Note that we don't take the product with $\cL og_\bk$ like we did for $\fM$. 
Let $\fC_V$ be the universal curve over $\fM_V$, i.e., the pullback from $  \mathscr{M} _{g_V, n_V\cup E_V}$. 
The natural perfect obstruction theory on $\mathscr{M}_V$ is given by
$$
(R\pi_{V,*} f_V^* \shT_{X_{r(V)}/\ubk})^\vee\ra \LL_{\underline{\mathscr{M}}_V/\underline{\fM_V}}  
$$
where $\pi_V:\shC_V\ra\mathscr{M}_V$ is the universal curve (also the pullback from $ \mathscr{M} _{g_V, n_V\cup E_V}$) and $f_V:\shC_V\ra X_{r(V)}$ the universal map and 
$(\cdot)^\vee$ means taking $\shR\shH{om}(\cdot,\shO)$.
We take the outer tensor product of these obstruction theories on $\prod_V\mathscr{M}_V$ via
$$
\mathbf{E}  := \boxtimes _V ( R\pi _{V,*} f^*_V \shT_{X_{r(V)}/\ubk})^\vee  
\quad \ra  \quad
    \mathbf{L}  := \boxtimes _V \mathbb{L}_{\underline{\mathscr{M}_V} /  \underline{\fM_V}} 
\cong \mathbb{L}_{\Pi _V \underline{\mathscr{M}_V} / \Pi _V \underline{\fM_V}}
$$
which is the natural perfect obstruction theory on $\prod_V\mathscr{M}_V$  relative to $ \prod_V \fM_V$, e.g.,~by \cite[Proposition 5.7]{BF}.
The isomorphism $\mathbf{L} \cong \mathbb{L}_{\Pi _V \underline{\mathscr{M}_V} / \Pi _V \underline{\fM_V}}$ between cotangent complexes can be seen 
by the distinguished triangles 
associated to towers of projections starting from $\mathscr{M}_{V_1}\times _{\ubk} \mathscr{M}_{V_2} \ra 
\mathscr{M}_{V_1}\times _{\ubk}\fM_{V_2} \ra \fM_{V_1} \times _{\ubk} \fM_{V_2}$ and base changes.


\subsection{Tangent sheaves and morphisms} \label{subsec-virt-classes}
The universal curve $\pi : \uC_\odot \ra \bigodot _V\mathscr{M}_V$ is obtained via pushout from the $\uC_V$ as in \eqref{glue-diag}.
The universal curve $\cC_{\Gamma}$ of $\mathscr{M} _{\Gamma}$ as well as various other universal curves are defined by the following fibre product diagram in log stacks
\[ \xymatrix{ \cC_{\Gamma}   \ar[r] \ar[d] & \fC_{\fM_{\Gamma}} \ar[r] \ar[d] & \fC _{  \mathscr{M}_{g, n}} \ar[d] & \cC_{\mathscr{M}} \ar[l] \ar[d]  \\
                     \mathscr{M}_{\Gamma} \ar[r] & \fM _{\Gamma} \ar[r] & \cL og_{  \mathscr{M} _{g,n}}  & \mathscr{M} \ar[l]  
                     }\]
starting from the log universal curve $\fC _{  \mathscr{M} _{g, n}} $ of $ \cL og_{  \mathscr{M}_{g,n}} $.
The underlying universal curve $\uC_{\Gamma}$ of $\mathscr{M} _{\Gamma}$ fits into the fibre product diagram
\[ \xymatrix{ {\uC}_{\Gamma} \ar[r]^{\tilde{\phi}_{\Gamma}} \ar[d]_{\pi _{\Gamma}} & {\uC}_\odot \ar[d]_{\pi} \\
     \underline{\mathscr{M}} _{\Gamma} \ar[r]_{\phi_{\Gamma}\quad }  & \bigodot _V \underline{\mathscr{M}} _V \ar@/_/[u]_{\iota _{e}} , } \]
     which defines $\tilde{\phi}_{\Gamma}$ that we also denote $\tilde{\phi}$. 
On $\uC_{\Gamma}$ we consider the following diagram with exact triangles as rows that we will prove to be commutative:
\begin{equation}
\vcenter{ \label{big-exact-diagram}
\resizebox{.95\textwidth}{!}{
\xymatrix@!0{
\TT_{\shC_{\Gamma}/\fC_{\fM_{\Gamma}}}  
    \ar[rrrrrr] \ar@{->}[dddddd] \ar[ddrrrr] & & & & && \oplus _{V}  \TT_{\shC_{\Gamma}/\fC_{\fM_{\Gamma}}} |_{\underline\shC_{V,\Gamma}} \ar[rrrrrr]
       \ar[ddrrrr]  & 
           & & & && \oplus _{e}  \TT_{\shC_{\Gamma}/\fC_{\fM_{\Gamma}}} |_{\mathscr{M}_{e}}
                \ar[ddrrrr]  &&& \\ \\
                       &&&& f^*\shT_{X/{\bf k}} \ \ar@{^{(}->}[rrrrrr] \ar@{-->}[dd] & &&  &&& \oplus _{V}  f^*\shT_{X/{\bf k}} |_{\underline\shC_{V, \Gamma}}  \ar@{->>}[rrrrrr] \ar[dd] 
                       & &&& && \oplus _{e}  f^*\shT_{X/{\bf k}} |_{\mathscr{M}_e}\ar[dd] \\ \\
          &&&& \tilde{\phi}^* \shE \ \ar@{^{(}->}[rrrrrr] && 
          &&&& \tilde{\phi} ^*  \oplus _{V} \iota _{V, *} \tilde{u}_V^* f^*_V\shT_{X_{r(V)}/{\ubk}}     \ar@{->>}[rrrrrr] &&
          &&&& \tilde{\phi} ^* \oplus _{e} \iota _{e, *} \ev_e^*\shT_{D/{\ubk}}  \\ \\
               \tilde{\phi}^*   \pi^* \TT _{\bigodot _V \mathscr{M} _V/ \prod _V \fM _V} \ar@{-->}[rrrruu]  \ar[rrrrrr]  
               &&&& &&  \tilde{\phi}^*  \oplus _V \pi ^*u^*_V \TT_{\mathscr{M} _V/\fM _V}  \ar[rrrrrr] \ar[rrrruu] 
              &&& &&&  \tilde{\phi}^* \pi^*\TT_u [1] \ar[rrrruu]
}}}
\end{equation}

Here, $\TT_{A/B}$ denotes the relative tangent complex of a map $A\ra B$, i.e., the dual of the relative cotangent complex $\LL_{A/B}$. 
For a strict map of fine log stacks $A\ra B$, we have $\LL_{A/B}=\LL_{\underline{A}/\underline{B}}=\LL_{\underline{A}/\cL og_B}$ by \cite{Ol-cotangent} and the map underlying every occurrence of $\TT$ in the diagram is strict.
Pullbacks of $\TT$ are well-defined, see \cite{Ol-cotangent}. 
We refer to the universal map $\cC_{\Gamma} \ra X$ by $f$ and
$\underline \shC_{V, \Gamma}$ is the closed substack of  $\underline \shC_{\Gamma}$ 
corresponding to the universal curve $\shC_{V}$ of $\mathscr{M} _V$. 
 There is a natural map $\underline \shC_{V, \Gamma} \ra \underline \shC_{V}$, which gives a log structure on
 $\uC_{V, \Gamma}$ by pullback. This will be denoted by $\shC_{V, \Gamma}$.
The image of the section $\underline{\mathscr{M} _{\Gamma}}\hra \ushC_{\Gamma}$ associated to the node $e$ is denoted by $\mathscr{M}_{e}$.

We now explain how to construct the diagram.
Let $\ushC_{V, \odot}$ denote the universal curve on $\bigodot _V \mathscr{M} _V$ corresponding to $\cC _V$. 
We denote by $\iota_V:\ushC_{V, \odot}\ra\ushC_\odot$ and $\iota_e:   \bigodot \mathscr{M} _V \ra\ushC_\odot$ the closed immersions attached to $V$ and $e$.

\begin{itemize}

\item 
The first two lines and the morphisms between them are obtained from the composition
$\TT_{\cC_{\Gamma} / \fC_{\fM _{\Gamma}}} \ra \TT _{\cC _{\Gamma}/\fM _{\Gamma}} \ra f^*\cT _{X/\bfk}$ 
by tensoring with the partial normalization exact sequence 
$$0\ra \shO_{\uC_{\Gamma}}\ra \bigoplus_V\iota _{V, *} \shO_{\uC_{V, \Gamma}} \ra \bigoplus_e \iota_e\shO_{e} \ra 0.$$

\item 
The maps $u_V,\tilde u_V$ are defined by the commuting diagram
     \[ \xymatrix{ \cC_{\Gamma}\ar[d]_{\tilde{\phi}} \ar@/^2pc/[rrr]_{f} & \ar[l]^{\iota _{V, \Gamma}} \cC_{V, \Gamma} \ar[d]_{\tilde{\phi}_V} & &  X \ar[r] & \uX    \\
     \cC_{\odot} \ar@/_1pc/[dr]_{\pi}  & \cC_{V, \odot} \ar[l]_{\iota _{V}} \ar[r]^{\tilde{u}_V} \ar@/^2pc/[d]^{\pi_V} & \cC_V \ar@/^2pc/[d]^{\pi_V} \ar[r]^{f_V}  & X_{r(V)} \ar[r] & \uX_{r(V)}\ar@{^{(}->}[u] \\
                       &      \ar[ul]_{\iota _e} \bigodot \mathscr{M}_V \ar[r]_{u_V} \ar[u]_{\iota _{V, e}} & \mathscr{M} _V \ar[u]_{\iota _{V, e} }         \ar[r]  &     D  \ar[ur]_{\iota_D}&   } \]
                       with the squares Cartesian and $\pi _V = \pi \circ \iota _V$. Here we double use notation for $\iota _{V,e}$ and $\pi_V$.

\item The forgetful morphisms $(D,\shM_{X}|_{D})\ra (D,\shM_{X_i}|_{D}) \ra D$ induce epimorphisms 
$\shT_{X/\bk}|_{D} \xrightarrow{\sigma_i} \shT_{X_i/\ubk}|_{D} \xrightarrow{\tau_i} \shT_{D/\ubk}=\shT_{D} $. 
On ${\ushC_\odot}$, we consider the sheaf epimorphism
\begin{equation}\\
\begin{array}{ccc}\label{Def_E_Ker}  
  \oplus_V \iota _{V,*} \tilde{u}_V^* f_V ^* \shT_{X_{r(V)}/\ubk}  &  \ra &  \oplus _e \iota _{e,*} \ev^*_e \shT_{D/\ubk} \\
  ( \xi_V )_V & \mapsto &  \sum _e  \tau_1(\xi _{V_1(e)}|_{D}) -  \tau_2(\xi _{V_2(e)}|_{D})  \end{array} 
\end{equation}
where $V_i(e)$ refers to the vertex $V$ of $e$ with $r(V)=i$.
We define $\cE$ as the kernel of  \eqref{Def_E_Ker}. This explains the third line of \eqref{big-exact-diagram} before taking pullback under $\tilde{\phi} :
\uC _{\Gamma} \ra \uC_{\bigodot}$. Now
the right two vertical morphisms 
from the second line to the third line are obtained from $f^*T_{X/\bfk}|_{\shC_{V,\Gamma}}  \ra \tilde{\phi}^*\iota _{V, *} \tilde{u}_V^* f_V^*T_{X_{r(V)}}$ and the adjunction transformation
$\id \Rightarrow  \iota _{e, *} \iota _{e}^* $. These two vertical morphisms uniquely determine the dashed arrow $f^*\cT_{X/\bfk} \dashrightarrow  \tilde\phi^*\cE $.

\item Recall $u$ from \eqref{gluediag}.
The bottom row of \eqref{big-exact-diagram} is obtained by applying $\tilde\phi^*\pi^*$ to the natural exact triangle
$$\TT _{\bigodot _V \mathscr{M} _V/ \prod _V \fM _V}\ra u^* \boxplus  _V  \TT _{\mathscr{M}_V /\fM _V}  \ra  \TT _{u} [1].$$
\item The bottom central diagonal up-arrow is $\tilde\phi^*$ applied to the composition
\begin{align*} 
\pi^*u_V^*\TT_{\mathscr{M} _V/\fM _V}\ra \iota_{V,*}\underbrace{\iota_V^*\pi^*}_{\pi_V^*}u_V^*\TT_{\mathscr{M} _V/\fM _V}&=\iota_{V,*}\tilde u_V^*\pi_V^*\TT_{\mathscr{M} _V/\fM _V}\\[-4mm]
&=\iota_{V,*}\tilde u_V^*\TT_{\cC_V/\cC_{\fM_V}}\ra \iota _{V, *} \tilde{u}_V^* f^*_V\shT_{X_{r(V)}/{\ubk}}
\end{align*} 
where the last map is $df_V$ and the second ``='' follows from the flatness of $\pi_V$ as it implies that the natural pullback morphism is an isomorphism.
\item The bottom right diagonal up-arrow is $\tilde\phi^*\pi^*$ of the natural map $\TT _{u} [1]\ra (\prod _e \ev_e)^* \TT _{\Delta} [1]$ (see \eqref{gluediag}) composed with the isomorphism $\TT _{\Delta} [1]\cong \oplus_e \cT _D$ and adjunction 
$$\pi^*\oplus_e \ev_e^*\cT_D\ra \oplus_e\iota_{e,*}\iota_e^*\pi^* \ev_e^*\cT_D=\oplus_e\iota_{e,*}\ev_e^*\cT_D.$$

\item The left long vertical map is the natural map $\TT_{\mathscr{M}_{\Gamma}/{\fM_{\Gamma}}}\ra\phi_\Gamma^* \TT_{\mathscr{M}_{\odot}/\fB}=\phi_\Gamma^*\TT_{\mathscr{M}_{\odot}/\Pi_V\fM_V}$ from \eqref{diag-gluing-moduli} using that $\mfs$ is \'etale (by Lemma~\ref{etale}, \eqref{wprime}), then applying $\pi^*$ and using flatness of $\pi$ to have $\TT_{\shC_{\Gamma}/{\fC_{\fM_{\Gamma}}}}=\pi^*\TT_{\mathscr{M}_{\Gamma}/{\fM_{\Gamma}}}$.
The commutativity of the hexagon that includes this map follows from the commutativity of \eqref{split-commutes}.

\item We prove the commutativity of the bottom right square.
From \eqref{gluediag}, we have a commuting diagram of cotangent complexes
\begin{equation}
\vcenter{ \label{for-bottom-right-square-commute}
 \xymatrix{ 
 \pi^* u^*\boxtimes _e  \ev_{e^{1,2}}^* \cT _{D\times D} \ar[r] &  \pi^* (\prod _e \ev_e)^* \TT _{\Delta} [1]  \ar[r]_{\quad\op{adj}} &  \oplus _e \iota_{e, *} \ev_e^* \cT_D   \\
 \pi^* u^* \boxtimes _V  \TT_{\mathscr{M}_V /\fM_V} \ar[r] \ar[u]  &    \pi^*  \TT_{u} [1] \ar[u]  &                           
} }
\end{equation}
where $\op{adj}$ is obtained by adjunction of $\pi^*\Rightarrow \iota _{e, *} \iota _e^* \pi^* = \iota _{e,*}$. 
By construction of the right diagonal up-arrow, the composition map  $\pi^* u^*\boxtimes _V  \TT_{\mathscr{M}_V /\fM_V} \ra \oplus _e \iota_{e, *} \ev_e^* \cT_D$ coincides with the the composition of the bottom horizontal map with this diagonal map. We need to show that this agrees with the other path that goes diagonal first and then horizontal after. To achieve this, we will factor the left vertical arrow in \eqref{for-bottom-right-square-commute}. We rewrite the top left corner using
$$\boxtimes _e  \ev_{e^{1,2}}^* \cT _{D\times D}=\boxtimes _{V\in e} \ev_{V,e}^*\shT_D= \boxtimes _{V\in e} \iota_{V,e}^*f_V^*\iota_{D,*}\shT_D$$
and consider the commutative diagram
\begin{align}\label{diff1}  
\vcenter{
\xymatrix{
f_V^*\cT _{X_{r(V)}} \ar[r]
  &  \iota_{V,e, *}\overbrace{\iota^*_{V,e} f_V^*\iota_{D,*}}^{\ev_{V,e}^*}\iota_D^*\cT _{X_{r(V)}}\ar^{\quad \quad \tau_{r(V)}}[r] 
  & \iota_{V,e, *}\ev_{V,e}^*\cT _{D}  \\ 
\TT _{\cC _V/\fC_{\fM_V}}  \ar[u]_{df_V} \ar[r]
   &   \iota_{V,e, *} \TT _{\mathscr{M}_V /\fM _V}  \ar[ru]_{ d\ev_{V,e}} \ar^{\iota_{V,e, *}\iota^*_{V,e}df_V}[u] } }
\end{align}
where the bottom left horizontal arrow is adjunction $\id\Rightarrow \iota _{V, e, *}\iota^* _{V, e}$ combined with
$\TT _{\cC _V/\fC_{\fM_V}} = \pi^*_V \TT _{\mathscr{M}_V /\fM _V}$ and $\iota^* _{V, e}\pi_V^*=\id^*$; the map $\tau_{r(V)}$ was defined above \eqref{Def_E_Ker}.
Since the left vertical map in \eqref{diff1} gives the bottom central diagonal arrow in \eqref{big-exact-diagram} and $\tau$ the horizontal one, we are done verifying the commutativity of the bottom right square.

\item The dashed diagonal arrow in \eqref{big-exact-diagram} is the unique map making the left right bottom corner commutative by the axioms of triangulated category applied to the bottom two rows as triangles and Lemma~\ref{Lemma:com_mono}.

\end{itemize}

\begin{lemma}\label{Lemma:com_mono}
Let $A\ra B$ be a monomorphism in an abelian category $\cA$ with enough injectives.
Let $C^{\bullet}$ be a complex in $\cA$ with $C^i=0$ for all $i <0$. Then the induced homomorphism 
$\Hom _{D^+(\cA )} (C^{\bullet}, A) \ra \Hom _{D^+(\cA )} (C^{\bullet}, B) $ is monic. 
\end{lemma}

\begin{proof} Let $\Kom (\cA)$ be the category of cochain complexes of $\cA$ and let $K(\cA)$ be the homotopy category of $\Kom (\cA)$.
First we find injective resolutions $I^{\bullet}$, $J^{\bullet}$ of $A$, $B$ respectively,
replacing $A \ra B$ by $I^\bullet \ra J^{\bullet}$ in $\Kom(\cA)$ with monomorphisms $I^i \ra J^i$ for all $i\ge 0$.
We have $$\Hom _{D^+(\cA )} (C^{\bullet}, A)  = \Hom _{K(\cA)} (C^{\bullet}, I^{\bullet} )  \ra 
 \Hom _{K(\cA)} (C^{\bullet}, J^{\bullet} ) = \Hom _{D^+(\cA )} (C^{\bullet}, B)  $$
 which is easily checked to be monic. 
\end{proof}

\subsection{Virtual Fundamental classes}
\label{subsec-virt-fund-cl}
Recall diagram \eqref{gluediag}. We define a natural perfect obstruction theory on $\bigodot _V \mathscr{M}_V$ relative to
$ \Pi _V \fM_{V}$ as follows. The virtual class
$[ \mathscr{M}_V ,  (R \pi _{V *} f^*_V \shT_{X_{r(V)}})^\vee]$
is obtained from the perfect obstruction theory
that comes from a chain of exact functors
\begin{eqnarray} (R\pi _* f^*_V\shT_{X_{r(V)}/\ubk})^\vee & \cong &  R\pi _* (f^*_V\shT_{X_{r(V)}/\ubk} ^\vee \ot \omega _{\pi} [1] )   \nonumber
\\ & \to & R\pi _* (\LL _{C_V/C_{\fM _V}} \ot \omega_{\pi} [1])        \nonumber 
\\ & \cong &  R\pi _* (L\pi ^* \LL _{\mathscr{M}_V/\fM _V} \ot \omega_{\pi} [1])   \label{def-functors} 
\\ & \cong &  \LL _{\mathscr{M}_V/\fM _V} \ot R\pi_* \omega _{\pi} [1]   \nonumber
\\ & \ra &  \LL_{\mathscr{M}_V/\fM _V}  \nonumber
\end{eqnarray}
where the first one is Grothendieck duality, the third one uses the fact that universal curve $\pi$ is flat, the fourth one is the projection formula,
and the last one is the trace map. We may think of \eqref{def-functors}, under the sequence of exact functors, as an output for a map of a pair $\shT_{X_{r(V)}/\ubk} \leftarrow \pi^*  \TT_{\mathscr{M}_V/\fM _V}$ as an input.

Note that $\omega_{\pi_V} =\iota_V^*\omega_\pi$, so if we apply the chain \eqref{def-functors} of exact functors for 
three pairs in the last two lines in diagram \eqref{big-exact-diagram},
we obtain a commuting diagram of exact triangles
\begin{equation}\label{comp_dist} 
\vcenter{
\xymatrix{ (\Pi _e \ev_e)^* \mathbb{L}_{\Delta} [-1] \ar[r]\ar[d] & u^* \mathbf{E}  \ar[r] \ar[d]
          & (R\pi _* \shE )^\vee \ar[d] \\
          \mathbb{L}_{u} [-1] \ar[r] & u^*  \mathbf{L}       \ar[r] 
          & \mathbb{L}_{\bigodot_V \mathscr{M}_V / \Pi _V \fM _V} } }
\end{equation}
where, by construction, the left vertical map is the natural one coming from the Cartesian square \eqref{gluediag}.
Also as in the construction, we use 
$$(R \pi_* (\oplus _e \iota _{e,*} \ev^*_e \shT_{D/\ubk}))^\vee\cong (\oplus _e \ev^*_e \shT_{D/\ubk} )^\vee \cong (\Pi _e \ev_e)^* \mathbb{L}_{\Delta} [-1].$$
Since the left vertical arrow in \eqref{comp_dist} is surjective at $h^0$, by the two four-lemmas (that are part of the usual five-lemma), $h^0$ of the right vertical arrow in \eqref{comp_dist} is an isomorphism and $h^{-1}$ is surjective. 
Thus, the right vertical arrow in \eqref{comp_dist} is an obstruction theory (see \cite[Def 4.4]{BF}).
We claim it is a perfect obstruction theory, i.e., $(R\pi_*\shE)^\vee$ is locally quasi-isomorphic to a complex of free sheaves in degree $-1$ and $0$. 
Equivalently, $R\pi_*\shE$ is locally quasi-isomorphic to $E_0\ra E_1$ (for $E_i$ free and in degree $i$). 
Indeed, we can take \eqref{Def_E_Ker} as a resolution for $\shE$, call this $F_0\ra F_1$.
We can replace $F_0\ra F_1$ by $[E_0\ra E_1]:=[F_0(\sum_j D_j)\ra F_1\oplus \bigoplus_j F_0|_{D_j}]$ for $D_j$ suitably chosen local sections of $\pi$ with $\sum_j D_j$ relatively ample so that $R\pi_*E_i$ is locally free and concentrated in degree $0$ for $i=0,1$, hence giving the perfectness.

By the functoriality of \cite[Proposition 5.10]{BF}, we conclude from \eqref{comp_dist}  that
\begin{align}\label{DeltaFund}            [\bigodot_V \mathscr{M}_V , (R\pi _*\shE)^\vee ] 
& = \Delta ^{!} \prod _V [ \mathscr{M}_V ,  (R \pi _{V *} f^*_V \shT_{X_{r(V)}})^\vee].
\end{align}

Now focus on the middle two lines in \eqref{big-exact-diagram} and note that $\shT_{X/\bk}|_{\uX_i}\ra\shT_{\uX_i/\ubk}$ is an isomorphism and so the kernel of the right vertical map is
$ \oplus _e \iota _{e,*} \ev_e^*N_{D/X_1}$ by the definition of $N_{D/X_1}$ in \eqref{log-normal-to-D}. The snake-lemma for this $2\times 3$-diagram gives the cokernel of the left vertical arrow, that is,
     there is a natural exact sequence 
     \[     0 \ra f^* \shT_{X/\mathbf{k}} \ra \tilde{\phi}^* \mathcal{E}  \ra \tilde{\phi}^*( \oplus _e \iota _{e,*} \ev_e^*N_{D/X_1}) \ra 0.\] 
If we apply the chain \eqref{def-functors} of exact functors for 
the left trapezoid in diagram \eqref{big-exact-diagram},
we obtain a commuting diagram of exact triangles
\begin{equation}
\vcenter{\label{two_relatives}  \xymatrix{ 
  \phi ^* ( R \pi _* \shE)^\vee \ar[r]  \ar[d] &  
  \mu _{\Gamma}^* \Big(R\pi _{*} f^* \shT_{X/\bk}\Big)^\vee \ar[r] \ar[d] &  
  \phi^* \Bigg(\bigoplus _{e\in E(\Gamma)} \ev^*_e N_{D/X_{1}}\Bigg)^\vee [1] \ar[d]  \\
  \LL _{\bigodot \mathscr{M}_V/\prod _V \fM _V} \ar[r]  &  \LL_{\mathscr{M} _{\Gamma}/\fM_{\Gamma}} \ar[r]  &   w^* \LL_{\fM_{\Gamma}/\prod _V \fM _V}[1].  }
}
\end{equation}

By Lemma~\ref{etale},\eqref{dis}, the top right corner is isomorphic to
$w^* \LL_{\fM_{\Gamma} /  \fB}[1]$ and since $\LL_{\prod _V \fM _V/\fB}=0$ by Lemma~\ref{etale},\eqref{wprime}, for the right vertical map to be isomorphic to the pullback of the natural map 
$\LL_{\fM_{\Gamma} /  \fB}\ra \LL_{\fM_{\Gamma}/\prod _V \fM _V}$, we need to prove that it is an isomorphism. 
Since the other two vertical arrows are perfect obstruction theories, the 4-lemma gives that the right vertical map is surjective. 
However, a surjective map of free sheaves of the same rank is an isomorphism (as this can be checked \'etale locally where they are projective).

Applying functoriality \cite[Proposition 5.10]{BF} to \eqref{two_relatives}, we conclude
\begin{equation}\label{BaseFund} 
[\mathscr{M}_{\Gamma}   / \Pi _V \fM _V ,  \phi ^*_{\Gamma} ( R \pi _* \shE)^\vee ] 
 = [\mathscr{M}_{\Gamma} / \fM _{\Gamma}, \mu _{\Gamma}^* (R \pi _{*} f^* \shT_{X/\bk})^\vee]. 
\end{equation}
Also, by a special case of functoriality for the \'etale map $\phi_\Gamma$,
\begin{align}\label{GlueFund}  [\mathscr{M}_{\Gamma}   / \Pi _V \fM_V ,  \phi^*_{\Gamma} ( R\pi _* \shE)^\vee ] 
 = \phi ^*_{\Gamma} [ \bigodot _V \mathscr{M}_V , (R \pi _* \shE)^\vee ] . 
\end{align}

\begin{proof}[Proof of Theorem~\ref{mainthm-cycle}] 
\label{mainproof}
The result is the composition of the identities Lemma~\ref{SplFund2}, \eqref{BaseFund}, \eqref{GlueFund} and \eqref{DeltaFund}.
\end{proof}

\noindent\rule{4cm}{0.4pt}

Conflict of Interest: None.

All data generated or analysed during this study are included in this published article.


\end{document}